\documentclass[a4paper]{amsart}
\usepackage[all]{xy}
\usepackage{amsthm,amsmath,amssymb,mathrsfs,graphicx, xcolor, tikz-cd}
\usepackage[latin1]{inputenc}
 \newtheorem{thm}{Theorem}
 
 \newtheorem{prop}[thm]{Proposition}
 \newtheorem{lemma}[thm]{Lemma}
 \newtheorem{klem}[thm]{Lemma}
 \newtheorem{cor}[thm]{Corollary}

 \theoremstyle{definition}
 \newtheorem{definition}[thm]{Definition}
 \newtheorem{ex}[thm]{Example}

 \theoremstyle{remark}
 \newtheorem{remark}[thm]{Remark}
\numberwithin{thm}{section}
\numberwithin{equation}{section}
\def\Spec{{\rm Spec}}

\def\Aut{{\rm Aut}}
\def\Gr{{\rm Gr}}
\def\GL{{\rm GL}}

\def\HN{{\rm HN}}
\def\BG{{$B_{{\rm dR}}^+$-Grassmannian }}

\def\HT{{\rm HT}}

\def\BB{{\rm BB}}
\def\HHom{{\mathcal{I}som}}

\def\can{{\rm can}}

\def\pr{{\rm pr}}
\def\ad{{\rm ad}}
\def\av{{\rm av}}

\def\Bun{{\rm Bun}}
\def\wa{{\rm wa}}
\def\a{{\rm a}}
\def\dom{{\rm dom}}

\def\Hom{{\rm Hom}}
\def\dim{{\rm dim}\,}

\def\Spa{{\rm Spa}}

\def\Spd{{\rm Spd}}

\def\sst{{\rm ss}}
\def\dR{{\rm dR}}

\def\Z{{\mathbb Z}}

\def\Qp{{F}}
\def\F{{\mathcal F}}
\def\E{{\mathcal E}}
\def\Q{{\mathbb Q}}

\def\O{{\mathcal O}}

\def\BGmub{B(G,\mu,b)}

\long\def\forget#1{}
\begin{document}

\begin{title}
{On Newton strata in the $B_{{\rm dR}}^+$-Grassmannian}
\end{title}

\author{Eva Viehmann}

\address{WWU M\"unster\\Mathematisches Institut \\ Einsteinstr. 62\\48149 M\"unster\\Germany}
\email{viehmann@uni-muenster.de}
\thanks{The author was partially supported by ERC Consolidator Grant 770936:\ NewtonStrat.}

\subjclass[2000]{
11G18 (14G20,   
14M15)}

\begin{abstract}
{We study parabolic reductions and Newton points of $G$-bundles on the Fargues-Fontaine curve and the Newton stratification on the \BG for any reductive group $G$. Let $\Bun_G$ be the stack of $G$-bundles on the Fargues-Fontaine curve. Our first main result is to show that under the identification of the points of $\Bun_G$ with Kottwitz's set $B(G)$, the closure relations on $|\Bun_G|$ coincide with the opposite of the usual partial order on $B(G)$. Furthermore, we prove that every non-Hodge-Newton decomposable Newton stratum in a minuscule affine Schubert cell in the \BG intersects the weakly admissible locus, proving a conjecture of Chen. On the way, we study several interesting properties of parabolic reductions of $G$-bundles, and determine which Newton strata have classical points.}
\end{abstract}
\maketitle

\tableofcontents

\section{Introduction}\label{intro}

In the past years, two main discoveries have revolutionized the field of $p$-adic Hodge theory and of arithmetic geometry: Scholze's construction of perfectoid spaces, and the definition of the fundamental curve of $p$-adic Hodge theory by Fargues and Fontaine. Since then, these have led to significant advances towards understanding local and global Shimura varieties, as well as local Langlands correspondences.

The main geometric objects in this theory are (moduli stacks of) $G$-bundles on the Fargues-Fontaine curve. We fix a prime $p$ and a connected reductive group $G$ over a finite extension $\Qp$ of $\mathbb{Q}_p$. Let $\Bun_G$ be the small v-stack assigning to every perfectoid space $S$ over $\overline{\mathbb{F}}_p$ the groupoid of $G$-bundles on $X_{S}$, compare \cite[Prop.~19.5.3]{SW}. Here, $X_{S}$ is the relative Fargues-Fontaine curve over $S$. For more details we refer to the later sections. We denote the underlying topological space of $\Bun_G$ by $|\Bun_G|$.  Let $C$ be an algebraically closed complete non-Archimedean field over $\overline{\mathbb F}_p$. By \cite{FGtorseurs}, we have a bijection $B(G)\rightarrow \Bun_G(C)$ where $B(G)$ is Kottwitz's set of Frobenius-conjugacy classes of elements of $G(\breve {\Qp})$. Here, $\breve {\Qp}$ denotes the completion of the maximal unramified extension of ${\Qp}$. We obtain a bijection $B(G)\rightarrow |\Bun_G|$, see \cite[Prop.~12.7]{SEt}. We denote the $G$-bundle corresponding to some $[b]\in B(G)$ by $\E_b$.

The set $B(G)$ carries a partial order which also describes the specialization order among $F$-isocrystals with additional structure in characteristic $p$, cf.~\cite{RR}. Using Kottwitz's classification \cite{Kottwitz1}, the elements $[b]\in B(G)$ are described by two invariants. The first is to associate with $b\in G(\breve {\Qp})$ its image under the Kottwitz map $\kappa_G$, an element of $\pi_1(G)_{\Gamma}$, where $\Gamma$ is the absolute Galois group of ${\Qp}$. The second invariant is the Newton point $\nu_b$ of $b$. If $G$ is quasi-split, this is an element of $X_*(T)^{\Gamma}_{\mathbb Q,{\rm dom}}$, where we refer to Section \ref{secnot} for the notation, and also for the analogue for non-quasi-split groups. In terms of these invariants, the partial order is then given by $[b]\leq [b']$ if $\kappa_G(b)=\kappa_G(b')$ and $\nu_{b'}-\nu_b$ is a non-negative rational linear combination of positive coroots. We equip $B(G)$ with the unique topology induced by the opposite of this partial order, i.e.~$[b'']\in \overline{\{[b']\}}$ if and only if $[b']\leq [b'']$, compare Corollary \ref{cortop}.

\begin{thm}\label{thmmain1}
The bijection $|\Bun_G|\rightarrow B(G)$ is a homeomorphism.
\end{thm}

As was pointed out to us by P. Scholze and D. Hansen, this result (together with a deep result of Fargues and Scholze on local charts for $\Bun_G$) immediately implies the following theorem, which proves a conjecture of Chen, \cite[Conj. 2.11]{Chen}. For details compare Section \ref{secoppfilt} below.

\begin{thm}\label{thmoppfilt}
Let $G$ be quasi-split and let $[b]\geq [b']\in B(G)$. Let $P$ be the parabolic subgroup of $G$ with Levi factor $M$ the centralizer of $\nu_{b}$ and such that $\nu_{b}$ is anti-dominant with respect to $P$. Let $b_M\in [b]\cap M(\breve F)$ with $M$-dominant Newton point $\nu_{b}$. Then $\E_{b'}$ has a reduction $(\E_{b'})_P$ to $P$ such that $(\E_{b'})_P\times^P M\cong \E_{b_M}^{M}$ is the $M$-bundle corresponding to the class of $b_M$ in $B(M)$.
\end{thm}

Several partial results towards Theorem \ref{thmmain1} have already been obtained in the past years: By results of Kedlaya-Liu \cite[Thm.~7.4.5]{KL} and Scholze-Weinstein \cite[Cor.~22.5.1]{SW}, $\E_{b''}\in \overline{\{\E_{b'}\}}$ implies $\nu_{b'}\leq \nu_{b''}$. In \cite[Thm.~III.2.7]{FS}, Fargues and Scholze prove that the map $\E_b\mapsto \kappa_G(b)$ is locally constant. In particular, the map in the theorem is continuous.

Hansen \cite{Hansen} proves the remaining assertion on openness for $G=\GL_n$, using Theorem \ref{thmoppfilt} for this case. Indeed, the main result of \cite{7authors} is that Theorem \ref{thmoppfilt} holds for $\GL_n$. To prove Theorem \ref{thmmain1}, Hansen starts with a vector bundle corresponding to $[b']$ together with a filtration such that the associated graded vector bundle corresponds to some $[b'']\geq [b']$. He then constructs a family of vector bundles with filtration with the same Newton polygon $\nu_{b'}$ that degenerates into a vector bundle where the associated filtration is split, allowing to conclude that it has the desired bigger invariant $[b'']$. However, the methods used in \cite{7authors} are tailored to the $\GL_n$-case, and it is not clear how to generalize them to other groups.

The idea for the proof of Theorem \ref{thmmain1} is completely different from Hansen's approach. We use families of $G$-bundles constructed as modifications of the trivial $G$-bundle via Beauville-Laszlo uniformization of $\Bun_G$. It turns out to be hard to compute the element in $B(G)$ corresponding to the modification associated with a given point in the $B_{\dR}^+$-Grassmannian $\Gr_{G}$. To circumvent this difficulty, we replace the computation of individual Newton points by that of the Newton point on a dense subset of a suitable semi-infinite cell in $\Gr_G$, see Theorem \ref{propkey}. Then we use the well-known closure relations between semi-infinite cells to conclude.

To explain our second main topic, let $C$ be a complete and algebraically closed extension of $\Qp$. Then we have the Cartan decomposition $$\Gr_G(C)=\coprod_{\{\mu\}}G(B_{\dR}^+(C))\mu(\xi)^{-1}G(B_{\dR}^+(C))/G(B_{\dR}^+(C))$$
where the union is over all conjugacy classes of cocharacters of $G$. It corresponds to a subdivision of $\Gr_G$ into locally spatial sub-diamonds, the affine Schubert cells $\Gr_{G,\mu}$, whose $C$-valued points are the contribution for the respective $\mu$ in the above decomposition.

Let $b\in G(\breve{\Qp})$ be a basic element, i.e. a minimal element of $B(G)$ with respect to the partial order. From the Beauville-Laszlo uniformization, we have for every point $x\in \Gr_{G,\mu}(C)$ a modification $\E_{b,x}$ of the $G$-bundle $\E_{b}$. Subdividing $\Gr_{G,\mu}$ according to the isomorphism class of $\E_{b,x}$ induces a decomposition of $\Gr_{G,\mu}$ into locally closed locally spatial sub-diamonds $\Gr_{G,\mu,b}^{[b']}$ called Newton strata. The set of $[b']\in B(G)$ such that the associated Newton stratum is non-empty is denoted $\BGmub$, for an explicit description compare Definition \ref{eqdefbgmub} and Corollary \ref{remnonempty}. For $b=1$, we obtain the classical Kott\-witz set $B(G,\mu,1)=B(G,-\mu_{\dom})$. Let $[b']\in \BGmub$ be the unique basic element. Then the corresponding Newton stratum  $\Gr_{G,\mu,b}^{[b']}$ is open in $\Gr_{G,\mu}$, and we call it the admissible locus. If $\mu$ is minuscule and $\kappa_G(b)=\mu^{\sharp}$, then $[b']=[1]$ and $\Gr_{G,\mu,b}^{[b']}$ can be identified with the admissible locus $\mathcal F(G,\mu,b)^a$ in the flag variety for $(G,\mu)$ in the sense of Rapoport and Zink \cite{RZ}, compare the constructions by Hartl \cite{Hartl} and Faltings \cite{Faltings}, for those pairs $(G,\mu)$ considered in loc.~cit.

It is a difficult and open question to describe the admissible locus. A first approximation is the weakly admissible locus constructed  by Rapoport and Zink \cite{RZ} and Dat-Orlik-Rapoport \cite{DOR}, who defined an open adic subspace $\mathcal F(G,\mu,b)^{\wa}$ of the adic flag variety associated with $G$ and $\mu$. In Section \ref{secwa} below we define the weakly admissible locus $\Gr_{G,\mu}^{\wa}$ for affine Schubert cells $\Gr_{G,\mu}$. For minuscule $\mu$, it coincides with the classical weakly admissible locus via the Bialynicki-Birula isomorphism $\Gr_{G,\mu}\rightarrow \F(G,\mu)^{\diamond}$. For general $\mu$ it is however not equal to the inverse image of the classical weakly admissible locus under the corresponding Bialynicki-Birula map. The weakly admissible locus is an open subspace of $\Gr_{G,\mu}$ containing the admissible locus, and the two spaces have the same classical points (i.e., points defined over any finite extension of $\breve F$), compare Theorem \ref{propwaa} below. However, even for $\mu$ minuscule the weakly admissible and the admissible locus only coincide in exceptional cases, for so-called fully Hodge-Newton decomposable pairs $(G,\mu)$, see \cite[Thm. 0.1]{CFS}. It is natural to ask for a description of the complement, i.e.~the intersection of the weakly admissible locus with the other Newton strata.

In this context, $[b']\in \BGmub$ is called Hodge-Newton decomposable if its Newton point satisfies $\nu_{b'}^{\sharp_L}=(\nu_b\mu^{-1}_{\dom})^{\sharp_L}\in \pi_1(L)_{\Gamma,\Q}$ for some proper Levi subgroup $L$ of the quasi-split inner form of $G$ containing the centralizer of $\nu_{b'}$. In Proposition \ref{prophonedec} we show that if $[b']$ is Hodge-Newton decomposable, then $L$ corresponds to the Levi subgroup of a parabolic subgroup $L$ of $G$ and every modification $\E_{b,x}$ for $x\in \Gr_{G,\mu,b}^{[b']}(C)$ has a reduction to ${}^{w_0}L$ (as modification, not only as modified bundle). We use this to show that Hodge-Newton decomposable Newton strata do not intersect the weakly admissible locus. Our second main result is that the converse also holds.

\begin{thm}\label{thmmain2}
Let $G$ be a connected reductive group over ${\Qp}$. Let $\{\mu\}$ be a minuscule conjugacy class of cocharacters of $G$, let $b\in G(\breve{\Qp})$ be basic and let $[b']\in \BGmub$. Then the following are equivalent.
\begin{enumerate}
\item $\Gr_{G,\mu,b}^{\wa}\cap \Gr_{G,\mu,b}^{[b']}\neq \emptyset$.
\item $[b']$ is Hodge-Newton indecomposable.
\end{enumerate}
\end{thm}

Several people have been working on this question before, and there are a number of partial results available. In \cite[Thm.~5.1]{Chen} Chen observes that  the proof of \cite[Thm.~6.1]{CFS} yields the following  (although the assertion in \cite{CFS} is slightly different):
Firstly, non-emptiness of $\Gr_{G,\mu,b}^{\wa}\cap \Gr_{G,\mu,b}^{[b']}$  implies Hodge-Newton indecomposability of $[b']$. In \cite[Conj.~5.2]{Chen}, Chen conjectures the assertion of Theorem \ref{thmmain2}, at least if $G$ is quasi-split and $\nu_b=\mu^{\sharp}$ as elements of $\pi_1(G)_{\Gamma,\Q}$. Chen, Fargues and Shen prove this assertion for Hodge-Newton indecomposable $[b']$ that are minimal in $\BGmub\setminus\{[b'_0]\}$ with respect to the usual partial order. Here, $[b'_0]$ denotes the unique basic element in $\BGmub$. Finally, Chen \cite[Prop. 5.3]{Chen} proves several particular cases for the group $G=\GL_n$ and explicit elements $[b']$. Shen \cite{Shen} proves a variant of the results of \cite{CFS} for non-minuscule $\mu$.

The approach we take to prove Theorem \ref{thmmain2} is quite the opposite of trying to generalize the existing proofs. Whereas previously, only elements close to the basic $[b'_0]\in \BGmub$ were considered, our main step is to prove the theorem for the maximal Hodge-Newton indecomposable element of $\BGmub$. We then use Theorem \ref{thmmain1} to deduce the general assertion.

Another interesting outcome of our study of the weakly admissible locus and its classical points is Theorem \ref{propclpt}, which characterizes all Newton strata $\Gr_{G,\mu,b}^{[b']}$ having classical points.\\

\noindent\emph{Acknowledgment.} I thank Miaofen Chen, Laurent Fargues, Paul Hamacher, David Hansen, Urs Hartl, Kieu Hieu Nguyen, Michael Rapoport and Timo Richarz for helpful discussions and David Hansen and Peter Scholze for making preliminary versions of \cite{Hansen2} resp.~of \cite{FS} available to me. I am grateful to Miaofen Chen, Qihang Li and the referees for pointing out some inaccuracies in the first version of this paper.

\section{Background}

\subsection{Notation}\label{secnot}
Let $F$ be a finite extension of $\mathbb Q_p$. We fix a uniformizer $\pi$ of $F$. 
Let $\breve{\Qp}$ be the completion of the maximal unramified extension of ${\Qp}$. Let $\Gamma$ denote the absolute Galois group of ${\Qp}$.

Let $G$ be a connected reductive group over ${\Qp}$ and $H$ a quasi-split inner form. Fix an inner twisting $G_{\breve{\Qp}}\overset{\sim}{\rightarrow} H_{\breve{\Qp}}.$

Let $A$ be a maximal split torus of $H$. Let $T$ be its centralizer, and let $B$ be a Borel subgroup of $H$ containing $T$. Let $U$ be its unipotent radical.

For quasi-split $G$ we denote by $W$ the Weyl group with respect to $T$, and by $w_0$ its longest element. For $w\in W$ and $\ast$ any element or subset of $G$, we denote by ${}^w\ast$ the conjugate of $\ast$ by $w$. 

Denote by $(X^*(T),\Phi,X_*(T),\Phi^{\vee})$ the absolute root datum, by $\Phi^+$ the positive roots and by $\Delta$ the simple roots of $T$  with respect to $B$.

Further, $(X^*(A),\Phi_0,X_*(A),\Phi_0^{\vee})$ denotes the relative root datum, $\Phi_0^+$ the positive roots and $\Delta_0$ the simple (reduced) roots.

Let
\begin{equation}\label{eqnewtonset}
\mathcal N(G)=(\Hom(\mathbb D_{\overline {\Qp}}, G_{\overline {\Qp}})/G(\overline {\Qp}))^{\Gamma}
\end{equation}
where $\mathbb D$ is the pro-torus with character group $\Q$, and where $G(\overline {\Qp})$ acts by conjugation. Then the inner twisting induces an identification $\mathcal N(G)=\mathcal N(H)=X_*(A)_{\mathbb Q,\dom}$.

On $X_*(A)_{\mathbb Q}$ resp.~$X_*(T)_{\Q}$ we consider the partial order given by $\nu\leq \nu'$ if $\nu'-\nu$ is a non-negative rational linear combination of positive relative resp.~absolute coroots. For this we do not assume that the elements are dominant.

We caution the reader that we use both the additive and the multiplicative notation for elements of $X_*(A)$ and $X_*(T)$, and sometimes switch from one to the other.

We denote by $B(G)$ the set of $G(\breve{\Qp})$-$\sigma$-conjugacy classes of elements of $G(\breve{\Qp})$. These are classified by two invariants, compare \cite{Kottwitz1}, \cite{Kottwitz2}, \cite{RR}. The first is the Kottwitz map $\kappa_G:B(G)\rightarrow \pi_1(G)_{\Gamma}$, where $\pi_1(G)=\pi_1(H)$ is the quotient of $X_*(T)$ by the coroot lattice. The second is the Newton map $\nu:B(G)\rightarrow \mathcal N(G)$. Then the map $$(\kappa_G,\nu):B(G)\rightarrow \pi_1(G)_{\Gamma}\times \mathcal N(G)$$ is injective. The two invariants are related by the condition that for every $[b]\in B(G)$, the image of $[b]$ under
$$B(G)\overset{\nu}{\rightarrow} \mathcal N(G)\cong X_*(A)_{\Q,\dom}\rightarrow X_*(T)_{\Q}\rightarrow \pi_1(G)_{\Gamma,\Q}$$ agrees with the image of $\kappa_G([b])$. Here, the above maps are the Newton map, the natural inclusion and the natural projection.

The set $B(G)$ has a partial order. It is defined by $[b]\leq [b']$ if $\kappa_G(b)=\kappa_G(b')$ and $\nu_b\leq \nu_{b'}$.

\begin{definition}
Let $[b]\in B(G)$ be basic and $\{\mu\}$ a conjugacy class of cocharacters of $G$. We write $\mu\in X_*(T)$ for the dominant representative. Then let
\begin{equation}\label{eqdefbgmub}
\BGmub:=\{[b']\in B(G)\mid \kappa_G(b')=\kappa_G(b)-\mu^{\sharp}, \nu_{b'}\leq \nu_b(\mu^{-1,\diamond})_{\dom}\}.
\end{equation}
 Here $\mu^{\diamond}$ is the Galois average of $\mu$ and $\mu^{\sharp}=\mu^{\sharp_G}$ is the image of $\mu$ in $\pi_1(G)_{\Gamma}$. These subsets inherit a partial order from the partial order on $B(G)$.
\end{definition}
\begin{remark}
\begin{enumerate}
\item For $[b]=[1]$ we obtain $B(G,\mu,1)=B(G,(-\mu)_{\dom})$.
\item In \cite[4]{CFS}, Chen, Fargues and Shen write ${B(G,\kappa_G(b)-\mu^{\sharp},\nu_b(\mu^{-1,\diamond})_{\dom})}$ for the set that we denote by $\BGmub$.
\item These subsets parametrize the non-empty Newton strata in a given affine Schubert cell, compare Corollary \ref{remnonempty} below.
 \end{enumerate}
\end{remark}

\subsection{The \BG}\label{secthebg}
We recall from \cite[3.4]{CS} and \cite[19]{SW} the \BG and its decomposition into affine Schubert cells.

For a perfectoid affinoid $\Qp$-algebra $(R,R^+)$ consider the surjective map $W_{\O_F}(R^{\flat,+})\rightarrow R^+$ and let $\xi\in W_{\O_F}(R^{\flat,+})$
be a generator of its kernel. Then we denote $B_{\dR}^+(R)$ the $\xi$-adic completion of $W_{\mathcal O_F}(R^{\flat,+})[1/\pi]$ and $B_{\dR}(R)=B_{\dR}^+(R)[\xi^{-1}]$.

The \BG $\Gr_G$ of $G$ over $\Spa~{\Qp}$ is the sheaf for the pro-\'etale topology representing the functor that maps any affinoid perfectoid ${\Qp}$-algebra $(R,R^+)$ to the set of pairs consisting of a $G$-torsor $\E$ on $\Spec B_{\dR}^+(R)$ and of a trivialization of $\E|_{\Spec B_{\dR}(R)}$. It is also the \'etale sheafification of the functor mapping a pair as above to $G(B_{\dR}(R))/G(B^+_{\dR}(R)).$

Let $C$ be an algebraically closed and complete extension of $\Qp$. Then we also write $B_{\dR}=B_{\dR}(C)$ and $B^+_{\dR}=B^+_{\dR}(C)$. Choosing an isomorphism $B_{\dR}^+\cong C [\![\xi]\!]$, the Cartan decomposition gives a disjoint decomposition $$\Gr_G(C)=\coprod_{\{\mu\}}G(B_{\dR}^+)\mu^{-1}(\xi)G(B_{\dR}^+)/G(B_{\dR}^+)$$ where the union is over all conjugacy classes of cocharacters of $G$.

Let $\{\mu\}$ be a conjugacy class of cocharacters of $G$, and let $E$ be its field of definition. We recall the affine Schubert cell associated with $\mu$. It is defined as the subfunctor $\Gr_{G,\mu}$ of $\Gr_{G,E}$ assigning to $S$ the set of all maps $S\rightarrow \Gr_G$ such that for all complete and algebraically closed $C$, and any $\Spa(C,{C}^+)\rightarrow S$, the corresponding element of $\Gr_G(C)$ lies in $G(B_{\dR}^+)\mu(\xi)^{-1}G(B_{\dR}^+)/G(B_{\dR}^+)$. Furthermore, let $\Gr_{G,\leq\mu}$ be defined similarly, using a union of $G(B_{\dR}^+)$-cosets over all $\mu'\leq\mu$. By \cite[19.2]{SW}, $\Gr_{G,\leq \mu}$ is a spatial diamond and proper over $\Spd~E$, and $\Gr_{G,\mu}$ is a locally spatial diamond which is open in $\Gr_{G,\leq \mu}$.

For the rest of Section \ref{secthebg} assume that $G$ is quasi-split and let $T,B,$ and $U$ be as above. Let $K$ be a complete field extension of $\Qp$ such that $G$ is split over $K$.
\begin{definition}
Let $\eta\in X_*(T)_{\dom}$. Let $S_{\eta,\eta}$ be the subfunctor of $\Gr_{G,(-\eta)_{\dom},K}$ assigning to a perfectoid space $S$ over $\Spa~K$ the set of maps $S\rightarrow \Gr_{G,(-\eta)_{\dom},K}$ such that each geometric point $\Spa(C',(C')^+)\rightarrow S$ corresponds to an element of $$U(B_{\dR}^+(C'))\eta(\xi)G(B_{\dR}^+(C'))/G(B_{\dR}^+(C'))\subset G(B_{\dR}^+(C'))\eta(\xi)G(B_{\dR}^+(C'))/G(B_{\dR}^+(C')).$$
\end{definition}
 \begin{prop}\label{propslambdaeta}
The map $S_{\eta,\eta}\rightarrow \Gr_{G,(-\eta)_{\dom},K}$ is an open immersion. In particular, $S_{\eta,\eta}$ is a locally spatial diamond. Furthermore, it is $\ell$-cohomologically smooth of dimension $\langle 2\rho,\eta\rangle$.
\end{prop}
Here, for the notions of $\ell$-cohomological smoothness and dimension we refer to \cite{SEt}. For the proof of Proposition \ref{propslambdaeta} we need some preparation.

\begin{remark}\label{remBB}
Let $E$ be the field of definition of the conjugacy class of $\mu$. We denote by $\F(G,\mu)$ the associated flag variety over $E$. Let $b\in G(\breve{\Qp})$.
By \cite[Prop.~19.4.2]{SW} (and \cite[3.4]{CS} for $G=\GL_n$) there is a natural Bialynicki-Birula map
$$\BB=\BB_{\mu} :\Gr_{G,\mu}\rightarrow \F(G,\mu)^{\diamond}.$$ If $\mu$ is minuscule, ${\rm BB}_{\mu}$ is an isomorphism, \cite[Prop.~19.4.2]{SW}. Let $C$ be a complete field extension of $\breve{\Qp}$. For $x\in \Gr_{G,\mu}(C)$ we have the following intrinsic description. We write $$x=x_1\mu^{-1}(\xi)G(B_{\dR}^+)/G(B_{\dR}^+)$$ with $x_1\in G(B_{\dR}^+)$. Then the class of $x_1$ in $G(B_{\dR}^+)/(G(B_{\dR}^+)\cap \mu^{-1}(\xi)G(B_{\dR}^+)\mu(\xi))$ is uniquely determined by $x$, and $\BB(x)$ is its image under the map to $\F(G,\mu)^{\diamond}(C)$ induced by the reduction $G(B_{\dR}^+)\rightarrow G(C)$.
\end{remark}
\begin{remark}
Let $G$ be quasi-split. We consider the decomposition of the flag variety $\F(G,\mu)$ into Schubert cells for the action of $B$. Let $P_{\mu}$ be the parabolic subgroup (containing $B$) corresponding to some $\mu\in X_*(T)_{\dom}$ and let $W_{\mu}$ be the corresponding subgroup of the Weyl group $W$ of $G$. Then $$\F(G,\mu)=\bigcup_{w\in W/W_{\mu}}\F(G,\mu)^w$$ with $\F(G,\mu)^w(C)=U(C) w P_{\mu}(C)/P_{\mu}(C)$.
\end{remark}

\begin{proof}[Proof of Proposition \ref{propslambdaeta}]
Since $\eta$ is assumed to be dominant, the defining condition on $C'$-points is equivalent to being an element of $G^1(B_{\dR}^+)U(B_{\dR}^+)\eta(\xi)G(B_{\dR}^+)/G(B_{\dR}^+)$ where $G^1(B_{\dR}^+)\subseteq G(B_{\dR}^+)$ is the kernel of the reduction modulo $\xi$. This in turn is equivalent to $\BB_{(-\eta)_{\dom}}(x)$ being in the open Schubert cell in the flag variety for $(-\eta)_{\dom}$. This implies the first assertion, and the second is an immediate consequence, since $\Gr_{G,(-\eta)_{\dom}}$ is a locally spatial diamond.

The assertion on smoothness and dimension follows from the same assertions for $\Gr_{G,\mu}$ in \cite[VI.2.4]{FS}.\end{proof}

The Iwasawa decomposition induces a decomposition
$$\Gr_G(C)=\coprod_{\lambda\in X_*(T)} U(B_{\dR})\lambda(\xi)G(B_{\dR}^+)/G(B_{\dR}^+).$$

\begin{definition}
Let $\lambda\in X_*(T)$.
\begin{enumerate}
\item For $\eta\in X_*(T)_{\dom}$ we define $S_{\lambda,\eta}$ to be the locally spatial sub-diamond $\lambda(\xi)\eta(\xi)^{-1}S_{\eta,\eta}\subset \Gr_{G,C}$.
\item For $\lambda\in X_*(T)$ let $S_{\lambda}$ be the subsheaf of $\Gr_{G,C}$ such that $S\rightarrow \Gr_{G,C}$ is in $S_{\lambda}$ if each geometric point $\Spa(C',(C')^+)\rightarrow S$ corresponds to an element of $$U(B_{\dR}(C'))\lambda(\xi)G(B_{\dR}^+(C'))/G(B_{\dR}^+(C')).$$
\end{enumerate}
\end{definition}

\begin{prop}
$S_{\lambda}$ is an ind-diamond, and
\begin{equation}\label{eqSle}
S_{\lambda}=\underset{\underset{\eta\in X_*(T)_{\dom}}{\rightarrow}}{\lim}S_{\lambda,\eta}
\end{equation}
as v-sheaves.
\end{prop}
\begin{proof}
Multiplying by $\lambda(\xi)^{-1}$ we may assume that $\lambda=1$. Let $S$ be affinoid perfectoid and consider a morphism $S\rightarrow S_{\lambda}$. Since $\Gr_G=\underset{\rightarrow_{\mu}}{\lim}\Gr_{G,\leq\mu}$, it induces a morphism $S\rightarrow \Gr_{G,\leq \mu}$ for some $\mu$. Since $\lambda=1$, the image of $\mu$ in $\pi_1(G)$ is trivial. We want to show that the above morphism factors through some $S_{\lambda,\eta}$. It is enough to show that for $\mu$ as above there is an $\eta\in X_*(T)_{\dom}$ such that every geometric point of $\Gr_{G,\leq \mu}$ that lies in $S_{\lambda}$ is in fact a geometric point of $S_{\lambda,\eta}$.

We choose a faithful representation $G\rightarrow \GL_n$ for some $n$, and may thus assume that $G=\GL_n$. Recall that we consider $\mu$ with trivial image in $\pi_1(G)$. Replacing $\mu$ by a larger element we may assume that it is of the form $((n-1)a,-a,\dotsc, -a)$ for some $a>0$. Then $x\in \Gr_{G,\leq \mu}(C)$ if and only if all entries of any representing matrix (with coefficients in $B_{\dR}(C)\cong C(\!(\xi)\!)$) have valuations greater than or equal to $-a$. The point $x$ lies in $S_{1}(C)$ if this representing matrix can be chosen in $U(B_{\dR})$. These conditions together imply that $x\in S_{1,\eta}(C)$ for $\eta=2a\rho^{\vee}$, i.e.~$\langle \alpha,\eta\rangle=a$ for every simple root $\alpha$.
\end{proof}
\begin{remark}\label{propmvne}
\begin{enumerate}
\item By \cite[Prop.~6.4]{Shen}, $S_{\lambda}$ is locally closed and $$\overline {S_{\lambda}}=\bigcup_{\lambda'\leq \lambda}S_{\lambda'}.$$  
\item From the Iwasawa decomposition we obtain $\Gr_{G,C}=\coprod_{\lambda}S_{\lambda}$.
\item If the intersection $S_{\lambda}\cap \Gr_{G,\leq\mu}$ is non-empty then the same proof as in the classical case shows that $\lambda_{\dom}\leq(-\mu)_{\dom}$.
\end{enumerate}
\end{remark}

\section{The Newton stratification and parabolic reductions}

\subsection{Modifications of $G$-bundles and Newton strata}\label{secmodgb}

We recall the construction of modifications of $G$-bundles and Newton strata in the adic flag variety. For more details, compare \cite[3.5]{CS}, \cite[4.2]{FQuelques} or \cite[III.3]{FS}.

Let $S$ be a perfectoid space in characteristic $p$. Then we have the associated relative Fargues-Fontaine curve, compare \cite[11]{SW}. It can be defined as $X_S=Y_S/\phi^{\Z}$ where for affinoid $S=\Spa(R,R^+)$, we have $$Y_S=\Spa~W_{\O_F}(R^+)\setminus \{[\varpi]\pi=0\}.$$
Here, $\varpi$ is a pseudo-uniformizer of $R$, and $\pi$ the chosen uniformizer of $F$. For $S=\Spa~(C_0,C^+_0)$ for a complete and algebraically closed non-archimedean field $C_0$ we also write $X$ instead of $X_S$.

A $G$-bundle on $X_S$ is defined as an exact tensor functor from the category of representations of $G$ to the category of vector bundles on $X_S$ or directly as a $G$-torsor locally trivial for the \'etale topology.

We denote by $\Bun_G$ the small v-stack of $G$-bundles on the Fargues-Fontaine curve. It assigns to every perfectoid space $S$ over $\overline{\mathbb F}_p$ the groupoid of $G$-bundles on $X_{S}$. For details we refer to \cite[3]{FS}. We denote the underlying topological space by $|\Bun_G|$.

By \cite{FGtorseurs} (and \cite[Prop.~12.7]{SEt}) we have a bijection between $|\Bun_G|$ and the set $B(G)$. To translate \cite[Def. 1.1]{FGtorseurs} into our terms, let $b\in G(\breve F)$. Let $\E_b$ be the $G$-bundle on $X$ obtained by taking the descent of the trivial $G$-torsor on $Y$ via the Frobenius map $(b\sigma)\otimes \phi$, where $\phi$ is the Frobenius on $Y$.

Let us also recall the algebraic Fargues-Fontaine curve. Let $\mathcal O(n)$ be the line bundle on $X$ for $b=\pi^{-n}$.
Let $P=\bigoplus_{n\geq 0}H^0(X,\mathcal O(n))$. The summand for some $n$ is equal to $\mathcal O(Y)^{\varphi=\pi^{n}}$. The algebraic curve is defined as $X^{{\rm cl}}={\rm Proj}(P)$. There is a morphism of ringed spaces $X\rightarrow X^{{\rm cl}}$ inducing an equivalence of categories between the categories of vector bundles on $X^{{\rm cl}}$ and on $X$, respectively, compare \cite{FQuelques}.

We recall the Beauville-Laszlo morphism from \cite[19]{SW} and \cite[III.3]{FS}. As in \cite[19]{SW}, $\Gr_{G}$ can be seen as the functor mapping any affinoid perfectoid $S=\Spa(R,R^+)$ over $\Spd~\Qp$ to the set of $G$-torsors over $\Spec(B_{\dR}^+(R^{\sharp}))$ together with a trivialization over $B_{\dR}(R^{\sharp})$. Here $R^{\sharp}$ is the unique untilt of $R$ corresponding to the map $S\rightarrow \Spd~\Qp$.

Let $\E\cong \E_b$ be a $G$-bundle over $X_{F^{\flat}}$, for some $b\in G(\breve{\Qp})$. The untilt $F$ corresponds to a point $\infty\in X_{F^{\flat}}$.  By \cite{FGtorseurs} and \cite[Thm.~6.5]{A}, $\E|_{X_{F^{\flat}}\setminus\{\infty\}}$ is trivial, and we tacitly always fix a trivialization.

For $S$ as above consider $x:S\rightarrow \Gr_G$. By \cite[Prop. 11.3.1]{SW}, $S$ is a closed Cartier divisor of $X_{S^{\flat}}$. Gluing $\E=\E_b$ over $X_{S^{\flat}}\setminus S$ and the trivial $G$-torsor over $\Spec(B_{\dR}^+(R))$ using the gluing datum given by $x$ \`a la Beauville-Laszlo, we obtain a $G$-bundle $\E_x=\E_{b,x}$. Mapping $x$ as above to $\E_x$ we obtain a canonical map $${\rm BL}_b:\Gr_{G,\mu}\rightarrow \Bun_G.$$

\begin{definition}
For $[b']\in B(G)$ let $\Gr_{G,\mu,b}^{[b']}$ be the subdiamond corresponding to $|{\rm BL}_b|^{-1}(\{\E_{b'}\})\subset |\Gr_{G,\mu}|$. It is called the Newton stratum (for the datum $(b,\{\mu\},[b'])$).
\end{definition}
As in \cite[3]{CS} this defines a decomposition of $\Gr_{G,\mu}$ into locally spatial locally closed subdiamonds.

\begin{remark}\label{rem32below}
\begin{enumerate}
\item Replacing $b$ by $g^{-1}b\sigma(g)$ for some $g\in G(\breve{\Qp})$ corresponds to a multiplication of the trivialization of $\E_b|_{X\setminus \{\infty\}}$. In particular, multiplication by $g^{-1}$ on $\Gr_G$ identifies Newton strata for modifications of $\E_b$ with Newton strata for modifications of $\E_{g^{-1}b\sigma(g)}$.
\item By \cite[4]{FGtorseurs}, if $\E\cong \E_{b'}$ for some $[b']\in B(G)$, then $-\kappa_G(b')$ is the first equivariant Chern class $c_1(\E)$ of $\E$. Let $b\in G(\breve F)$ and $x\in \Gr_{G,\mu,b}^{[b']}(C)$. By \cite[Lemma 3.5.5]{CS} we have
\begin{equation}\label{eqmodkappa}
-\kappa_G(b')=c_1^G(\E_{b,x})=\mu^{\sharp}+c_1^G(\E_b)=\mu^{\sharp}-\kappa_G(b).
\end{equation}
\end{enumerate}
\end{remark}
In this work we are mainly interested in Newton strata for modifications of the $G$-bundle associated with a basic element $b$. In Corollary \ref{remnonempty} below we show that in this case a Newton stratum $\Gr_{G,\mu,b}^{[b']}$ is non-empty if and only if $[b']\in \BGmub$.

\begin{definition}
Let $[b']\in \BGmub$ be the unique basic element. Then the corresponding Newton stratum $\Gr_{G,\mu,b}^{\a}=\Gr_{G,\mu,b}^{[b']}$ is called the admissible locus (for the datum $(b,\{\mu\})$). 
\end{definition}

\subsection{Adjoint quotients and inner forms}\label{secadinn}

In this subsection we collect some reduction steps and comparisons that we need later on to reduce proofs to the case of quasi-split $G$.

\subsubsection{} Let $G$ be a connected reductive group over $\Qp$. Let $G_{\ad}$ be its adjoint group. By a subscript \emph{ad} we denote images under the projection $\pr:G\rightarrow G_{\ad}.$

By \cite[Lemma III.2.10]{FS}, the induced morphism $\pr:\Bun_G\rightarrow \Bun_{G_{\ad}}$ is a surjective map of v-stacks. By \cite[Prop.~12.9]{SEt}, the corresponding map $|\pr|:|\Bun_G|\rightarrow |\Bun_{G_{\ad}}|$ is thus a quotient map.

Let $b\in G(\breve {\Qp})$ be basic and $\{\mu\}$ a conjugacy class of cocharacters of $G$. Since the value of $\kappa_G$ is constant on $\BGmub$, the projection $\BGmub\rightarrow B(G_{\ad},\mu_{\ad},b_{\ad})$ is injective.

 Furthermore, passing to the adjoint group commutes with the Beauville-Laszlo map. In particular, for each $[b']\in \BGmub$, the morphism $\pr:\Gr_G\rightarrow \Gr_{G_{\ad}}$ restricts to a map $\pr:\Gr_{G,\mu,b}^{[b']}\rightarrow \Gr_{G_{\ad},\mu_{\ad},b_{\ad}}^{[b'_{\ad}]}$.

\subsubsection{} \label{remcompgj}
 Let $b_0\in G(\breve{\Qp})$  be basic. Let $G_{b_0}$ be the inner form of $G$ obtained by twisting with $b_0$.

We obtain an isomorphism $\Bun_{G}=\Bun_{G_{b_0}}$ by mapping a $G$-bundle $\E$ on $X_S$ to the $G_{b_0}$-torsor of isomorphisms of $G$-torsors $\HHom(\E_{b_0},\E)$. On points, it induces a bijection $B(G)\cong B(G_{b_0})$ sending $[b_0]$ to $[1]$. In terms of our above construction of $\E_b$ we can make this more explicit. Let $C$ be again a complete and algebraically closed extension of ${\Qp}$ and let $b\in G(C)$. Then $\E_b^G$ is obtained by descending the trivial $G$-bundle on $Y$ via the Frobenius that is twisted by $b\sigma$. We map it to the descent of the trivial $G_{b_0}$-bundle on $Y$ via the Frobenius twisted by $(bb_0^{-1})(b_0\sigma)$ where $b_0\sigma$ is the Frobenius map on $G_{b_0}$. In other words, the above bijection maps $\E_b^G$ to $\E_{bb_0^{-1}}^{G_{b_0}}$. This map changes the Newton point of each class by $\nu_{b_0}$, which is central, and the Kottwitz point by adding $\kappa_G(b_0)$, compare \cite[3.4]{Kottwitz2}. In particular, the bijection between $B(G)$ and $B(G_{b_0})$ is compatible with the partial orders on the two sets.

Let $\{\mu\}$ be a conjugacy class of cocharacters of $G$, and assume that $b\in G(\breve{\Qp})$ is basic. Restricting the above map, we also obtain a bijection
\begin{equation*}
B(G,\mu,b)\cong B(G_{b_0},\mu,bb_0^{-1}).
\end{equation*}

Using the definition of $G_{b_0}$ as an inner form of $G$, the identity map $G_C\rightarrow G_{b_0,C}$ induces an isomorphism $\Gr_{G,C}\rightarrow \Gr_{G_{b_0},C}$. It identifies corresponding Newton strata $\Gr_{G,\mu,b}^{[b']}$ and $\Gr_{G_{b_0},\mu,bb_0^{-1}}^{[b'b_0^{-1}]}$ for all $[b']\in \BGmub$.

\subsection{Parabolic reductions}

\subsubsection{Slope vectors and non-positivity}\label{secsvnp}
Let $G'$ be a parabolic or Levi subgroup of $G$. Let $\E$ be a $G$-bundle on $X$. Then a reduction of $\E$ to $G'$ is a $G'$-bundle $\E_{G'}$ on $X$ together with an isomorphism $\E_{G'}\times^{G'}G\overset{\sim}{\longrightarrow}\E$.

Let $P$ be a parabolic subgroup of $G$, let $M$ be the Levi quotient of $P$ and let $N$ be its unipotent radical.

\begin{definition}\label{defslope}
Let ${\E}$ be a $G$-bundle on $X$ and ${\E}_P$ a reduction to $P$. Then the map
\begin{align*}
X^*(P)&\rightarrow \Z\\
\chi&\mapsto \deg(\chi_*{\E}_P)
\end{align*}
is Galois invariant and can thus be viewed as an element $v({\E}_P)\in \pi_1(M)_{\Q,\Gamma}= X_*(Z_M)_{\Q}^{\Gamma}$, called the slope vector of the reduction.

In case that $G$ is quasi-split and that $P$ is standard, we can also view $v(\E_P)$ as an element of $X_*(T)^{\Gamma}_{\mathbb Q}=X_*(A)_{\Q}$ that is central in $M$.

We call the reduction non-positive if for each $P$-dominant character $\chi\in X^*(P/Z_G)$ we have $\deg \chi_*(\E_P)\leq 0.$
\end{definition}
\begin{remark}
\begin{enumerate}
\item Assume that $G$ is quasi-split and that $P$ is standard. Then non-positivity is a weaker condition than requiring that the slope vector of the reduction is anti-dominant. For example, consider $G=GL_n$. Then a reduction of a bundle to a standard parabolic subgroup $P$ with some slope vector $v$ is non-positive if and only if all breakpoints of the polygon corresponding to $v$ lie below the straight line joining the endpoints of the polygon. 
\item In order to switch between the conventions for slope vectors of parabolic reductions of $G$-bundles and for the Newton vectors of the associated elements of $B(G)$ we introduce the following notation: For $\nu$ in $X_*(A)_{\Q}$ or $X_*(T)_{\Q}$ let $\nu^*=w_0(-\nu)$ where $w_0$ is the longest Weyl group element. 

\item Assume that $G$ is quasi-split. If $\E\cong \E_{b'}$ for some $[b']\in B(G)$, let $P$ be the parabolic subgroup associated with $v=\nu_{b'}^*$. Then $\E$ has a unique reduction to $P$ of slope vector $v$, called the canonical reduction of $\E$, see \cite{FGtorseurs}. The corresponding slope polygon $v=v({\E}_P)=\nu_{b'}^*$ is the Harder-Narasimhan polygon of ${\E}$.

Let $\E_{b'}$ be a $G$-bundle on $X$ and $\E_P$ any reduction to a parabolic subgroup of $G$. Then the comparison theorem for the Harder-Narasimhan reduction implies that the slope vector $v$ of $\E_P$ satisfies $v\leq \nu_{b'}^*$.
\end{enumerate}
\end{remark}

\begin{prop}\label{lemlamlem}
Let $G$ be a quasi-split reductive group over $\Qp$ and let $P$ be a standard parabolic subgroup. Let $M$ be its Levi quotient, and fix an embedding $M\subset P$. Let $b\in M(\breve {\Qp})$. Let $\E=\E_{b}^G$ be the associated $G$-bundle over $X$. Let $\E_P$ be any reduction of $\E$ to $P$. Then the following are equivalent:
\begin{enumerate}
\item $\E_P\times^P M\cong \E_{b}^M$
\item $\E_P\cong \E_{b}^P$
\item There is an automorphism of $\E$ identifying the subtorsor $\E_P$ with $\E_b^P\hookrightarrow \E_b^G$.
\end{enumerate}
\end{prop}

\begin{definition}
Let $G$ be reductive over $\Qp$ and let $P$ be a parabolic subgroup of $G$. Let $\E$ be a $G$-bundle on $X$ and let $\E_P$ be a reduction to $P$. We call the reduction $\E_P$ split if $(\E_P\times^P M)\times^M G\cong \E$.
\end{definition}
\begin{remark}
Let $\E$ be a $G$-bundle on $X$ and let $\E_P$ be a reduction to $P$. Let $b_M\in M(\breve {\Qp})$ with $\E_P\times^P M\cong \E_{b_M}^M$. Then the reduction $\E_P$ is split if and only if $\E\cong \E_{b_M}^G$. This is then also equivalent (by Proposition \ref{lemlamlem}) to conditions (2) or (3) of the proposition. Notice that we do not make any dominance assumption on the slope vector of the reduction.
\end{remark}

\begin{proof}[Proof of Proposition \ref{lemlamlem}]
Clearly, $(3)\Rightarrow (2)\Rightarrow (1)$. For the converse, we first consider the case of $\GL_n$. Using the equivalence between $\GL_n$-torsors and vector bundles, we are in the following situation: We have a vector bundle $\F$ of rank $n$ and a decomposition $\F=\F_1\oplus\dotsm\oplus \F_r$ for some $r$. Further we have a filtration $\F_0'=(0)\subseteq\F_1'\subseteq \dotsm\subseteq \F_r'=\F$ with $\F_i'/\F_{i-1}'\cong \F_i$. We have to show that there is an automorphism of $\F$ mapping $\F_i'$ to $\bigoplus_{j\leq i}\F_i$. We may assume that all $\F_i$ are stable by refining the filtration and the decomposition using a decomposition of each $\F_i$ into stable subbundles. Using descending induction on $r$ it is enough to show that there is an inclusion $\F_r\hookrightarrow \F$ inducing an isomorphism $\F\cong\F_{r-1}'\oplus \F_r$. Let $\lambda_r$ be the slope of the (stable) vector bundle $\F_r$. Let $\tilde \F$ be the filtration step of the Harder-Narasimhan filtration of $\F_{r-1}'$ of Harder-Narasimhan slopes $\geq \lambda_r$. We view it as a subbundle of $\F$ whose Harder-Narasimhan vector is the initial part of the Harder-Narasimhan vector of $\F$ consisting of all slopes $\geq \lambda_r$ except for the part corresponding to $\F_r$. Since the Harder-Narasimhan filtration as well as refinements having the same slope polygon are split for vector bundles on $X$, we have $\F\cong \tilde \F\oplus \F'$ for some $\F'$ and similarly for $\F_{r-1}'$. Replacing $\F$ by $\F'$ we may assume that all slopes of $\F$ are less than or equal to $\lambda_r$. Then we have a canonical map $\F_r\hookrightarrow \F$ as the first filtration step of the Harder-Narasimhan filtration of $\F$. The composition with the projection to $\F_r$ maps $\F_r$ to a subbundle of $\F_r$ which is also a quotient of $\F_r$. Since $\F_r$ is stable, it is either equal to $\F_r$ (in which case we have constructed the desired section of $\F\rightarrow \F_r$) or trivial. However, in the latter case, the map $\F_r\rightarrow \F$ would have image in $\F_{r-1}'$, in contradiction to our assumption that all Harder-Narasimhan slopes of $\F_{r-1}'$ are smaller than $\lambda_r$. Thus we obtain an induced isomorphism $\F\cong \F_{r-1}'\oplus \F_r$, which by induction implies the proposition for $G=\GL_n$.

Now we consider the general case. Again we can refine the parabolic reduction $\E_P$ using a refinement of the canonical reduction of $\E_P\times^P M$ to assume that $\E_P\times^P M$ is a stable $M$-bundle. Let $\HN(\E)$ denote the Harder-Narasimhan polygon of $\E$. Let $\E_{\can}$ be the canonical reduction of $\E$, a reduction to the standard parabolic subgroup $P_0$ of $G$ corresponding to $\HN(\E)$. Let $w$ be the shortest  representative of the unique class in $W_P\backslash W/W_{P_0}$ with $w(\HN(\E))= \HN(\E_P\times^P M)$. Then the parabolic reductions $\E_P$ and $\E_{\can}$ are in constant relative position $w$ over all of $X$, where the relative position is defined as in \cite[4]{S}. Indeed, this can be checked on a suitable representation $G\rightarrow \GL_n$, where it follows from (3) for $\GL_n$. Intersecting the two reductions of $\E$ we obtain a sub-$Q$-bundle $\E_Q$ where $Q=P\cap {}^wP_0\cong {}^{w^{-1}}P\cap P_0$. Let $Q_u$ be the unipotent radical of $Q$. Then from the explicit description of $w$ (and the fact that $\E_P\times^P M$ is stable), we obtain $Q/Q_u\cong M$ and $\E_Q\times^Q Q/Q_u\cong \E_P \times^P M$ is a reduction of $\E$. Viewing $\E_Q$ as a subtorsor of $\E_{\can}$, we can now apply the same argument as in \cite[Proof of Prop.~5.16]{FGtorseurs} (which also works for refinements of the canonical reduction since $H^1(X,\mathscr{U})=0$ for any vector bundle $\mathscr U$ on $X$ whose Harder-Narasimhan slopes are all $\geq 0$). We obtain that $\E_Q\cong (\E_{Q}\times^Q M)\times^M Q$ where we choose an embedding of $M$ into $Q$ that is a section of the projection map. Altogether we obtain subtorsors $\E_M\hookrightarrow \E_Q\hookrightarrow \E_P\hookrightarrow \E$. From Harder-Narasimhan theory together with the explicit description of the automorphisms of $\E$ in \cite[III.5.1]{FS}, we obtain that $\E_M$ (as subtorsor) is obtained from $\E_b^M$ by an automorphism of $\E$ which then also identifies $\E_P$ and $\E_b^P$.
\end{proof}

\subsubsection{Inner forms}\label{sec332}

Let $b\in G(\breve{\Qp})$.

We consider a parabolic subgroup $P$ of $G_{\breve{\Qp}}$ that is stable under $b\sigma$. Then we obtain a $P$-bundle over $X_C$ by descending the trivial $P$-bundle on $Y_C$ via the Frobenius given by $(b\sigma)\otimes\varphi$. It can be seen as a reduction $\E_{b,P}$ of $\E_b$ to $P$ in the generalized sense that we do not require $P$ to be defined over $F$. 

Let $b_0\in G(\breve{\Qp})$ be basic and such that $P$ as above is also stable under $b_0\sigma$. Then the parabolic subgroup $P$ can also be seen as the base change to $\breve{\Qp}$ of a parabolic subgroup $P_{b_0}$ of $G_{b_0}$ (defined over ${\Qp}$). The above reduction $\E_b^P$ corresponds to a reduction $(\E_{bb_0^{-1}}^{G_{b_0}})_{P_{b_0}}$ of $\E_{bb_0^{-1}}^{G_{b_0}}$ to $P_{b_0}\subset G_{b_0}$ in the sense of Section \ref{secsvnp}.

Two special cases are of particular interest: If $b$ is basic, we can apply these considerations to $b_0=b$. If $P$ is defined over ${\Qp}$, we can take $b_0=1$.

\begin{remark}\label{rem332}
Let $b_0,b_1$ be basic such that $P$ is stable under $b_0\sigma$ and under $b_1\sigma$. From the definition of being non-positive, we obtain that the reduction $(\E_{bb_0^{-1}}^{G_{b_0}})_{P_{b_0}}$ is non-positive if and only if the same holds for $(\E_{bb_1^{-1}}^{G_{b_1}})_{P_{b_1}}$.

Indeed, the condition depends on pull-backs of the bundles under characters of $P$ which are invariant under $b_0\sigma$ resp.~under $b_1\sigma$. These two conditions are equivalent since $b_0b_1^{-1}$ stabilizes $P$ and is hence contained in $P$.
\end{remark}

\subsubsection{Modifications}

\begin{remark}
Recall (for example from \cite[Lemma 2.4]{CFS}) the following comparison between parabolic reductions of modifications. Let $\E\cong \E_b$ be a $G$-bundle on $X$ and let $\E'$ be the modification of $\E$ associated with some trivialization of $\E|_{X\setminus\{\infty\}}$ and some $x\in \Gr_{G}(C)$. Then the isomorphism between $\E|_{X\setminus\{\infty\}}$ and $\E'|_{X\setminus\{\infty\}}$ induces for every $P$ a bijection
\begin{equation}\label{rem24}
\{\text{reductions of } \E \text{ to }P\}\rightarrow\{\text{reductions of } \E' \text{ to }P\}.
\end{equation}
On the other hand, a reduction of some $b\in G(\breve{\Qp})$ to a parabolic subgroup $P$ or to its Levi subgroup $M$ is defined (following \cite[Def.~2.5]{CFS}) as an element $b'\in [b]\cap P(\breve{\Qp})$ resp. $b'\in [b]\cap M(\breve{\Qp})$ together with some $g\in G(\breve{\Qp})$ with $b'=g^{-1}b\sigma(g)$, up to equivalence. Here, $(b_M,g)\sim (h^{-1}b_M\sigma(h),gh)$ for any $h\in M(\breve{ \Qp})$.

The notions of parabolic reduction of some $b\in G(\breve \Qp)$ and of the associated $G$-bundle $\E_b$ do not correspond to each other. A reduction of $b$ to $P$ induces a natural reduction of $\E_b$ to $P$, but not conversely. In particular, the analog of \eqref{rem24} for reductions of $b$ would be clearly false.
\end{remark}

Considering modifications of $P$-bundles, and of the associated $M$-bundles, we obtain the following easy, but very useful observation.

\begin{klem}\label{lemkey}
 Let $\E_M$ be an $M$-bundle on $X$ together with a trivialization of its restriction to $\Spec(B_{\dR}^+)$, and let $\E_P=\E_M\times^M P$ and $\E=\E_M\times^M G$. Let $x\in \Gr_G(C)$. From the Iwasawa decomposition $G(B_{\dR})=P(B_{\dR})G(B_{\dR}^+)$ we obtain a representative $x_0$ of $x$ in $P(B_{\dR})$. Let $\pr_M(x)$ be its (well-defined) image in $\Gr_M(C)$. Then by the previous remark, $\E_P$ induces a reduction $(\E_x)_P$ of $\E_x$ to $P$. It coincides with the modification of $\E_P$ corresponding to $x_0$. Furthermore, we have $$((\E_P)_{x_0})\times^P M\cong (\E_M)_{\pr_M(x)}.$$
\end{klem}

\begin{lemma}\label{lemslopekappa}
We use the notation of Lemma \ref{lemkey}.
\begin{enumerate}
\item The slope vector $v((\E_x)_P)$ coincides with the image of $c_1^M(\E_P\times^P M)- \pr_M(x)^{\sharp_M}\in \pi_1(M)_{\Gamma}$ in $\pi_1(M)_{\Gamma,\Q}$. Here, $\pr_M(x)^{\sharp_M}$ denotes the image of $\pr_M(x)\in \Gr_M$ under the projection to $\pi_1(M)_{\Gamma}$.
\item Assume that $G$ is quasi-split and that $x\in S_{\lambda}(C)$ for some $\lambda\in X_*(T)$. Let $\E_1^B$ be the trivial $B$-bundle, and $\E=\E_1^B\times^B G$. Then the slope vector of the induced reduction $(\E_x)_B$ is $-\lambda^{\diamond}$.
\item Let $P\subseteq P'\subseteq G$ be two parabolic subgroups, and let $\E_P$ be a $P$-bundle over $X$ with slope vector $v_P$. Then the slope vector $v_{P'}$ of $\E_P\times^P P'$ is the image of $v$ under the projection map $\pi_1(M)_{\Q,\Gamma}\rightarrow \pi_1(M')_{\Q,\Gamma}.$
\end{enumerate}
\end{lemma}
\begin{proof}
The first assertion follows from Lemma \ref{lemkey} together with \eqref{eqmodkappa}.

The second assertion follows from the first, using that in this case $Z_M=M=T$, and that the map to $X_*(T)_{\Q}^{\Gamma}\cong \pi_1(G)_{\Gamma,\Q}$ maps $\lambda$ to $\lambda^{\diamond}.$

The third assertion is obvious.
\end{proof}

For later use we consider the following application.
\begin{cor}\label{lemadomspl}
Assume that $G$ is quasi-split. Let $x\in \Gr_{G,\mu,1}^{[b']}$ for some $[b']$. Let $P$ be a standard parabolic subgroup of $G$, and let $M$ be its standard Levi factor. Assume that the reduction $(\E_{1,x})_P$ corresponding to the reduction $\E_1^P$ of $\E_1$ is split, $(\E_{1,x})_P\times^P M$ is semistable, and the slope vector of $(\E_{1,x})_P$ is $P$-regular and anti-dominant. Then $x\in P(\Qp)M(B_{\dR})G(B_{\dR}^+)/G(B_{\dR}^+)$.
\end{cor}
\begin{proof}
Let $(\E_{1,x})_M$ be an $M$-subbundle of $(\E_{1,x})_P$ with $(\E_{1,x})_P\cong (\E_{1,x})_M\times^M P$. Since $(\E_{1,x})_M\times^M P$ is semi-stable, the slope vector of $(\E_{1,x})_P$ coincides with the Harder-Narasimhan vector of $(\E_{1,x})_P\times^P M$. The slope vector  is $P$-regular and anti-dominant, thus $(\E_{1,x})_M$ is uniquely determined as the intersection of $(\E_{1,x})_P$ with the canonical reduction of $\E_{1,x}$, as in the proof of Proposition \ref{lemlamlem}. Using the Iwasawa decomposition $G(B_{\dR})=P(B_{\dR})G(B_{\dR}^+)$ we find a representative $x_0$ of $x$ in $P(B_{\dR})$, and denote by $\pr_M(x)$ its (well-defined) image in $\Gr_M(C)$. From the above modification we obtain that $(\E_{1,x})_M\cong (\E_{1,x})_P\times^P M={\E_{1,\pr_M(x)}^M}$. We now compare the modification between $\E_1^P$ and $(\E_{1,x})_P$ given by $x$ to the modification between $\E_1^M$ and $(\E_{1,x})_M$ given by $\pr_M(x)$, and the associated modification of the $P$-bundles obtained by taking a pushout to $P$: Let $\tilde x\in M(B_{\dR})$ describe a modification of $M$-bundles that is inverse to $\pr_M(x)$. It induces a modification between $(\E_{1,x})_M\times^M P$ and a $P$-torsor $\E'_P$ on $X$ with $\E_P'\times^P M\cong \E_1^M$. By \cite[Cor.~2.9]{Chen}, this implies that $\E'_P$ is a split reduction of $\E_1^G$ to $P$. In other words, the two inverse modifications corresponding to $\pr_M(x)$ (as modification of $(\E_{1,x})_M\times^M P\cong (\E_{1,x})_P$) and $x$ coincide up to an automorphism of $\E_1^P$. Thus there is an element $g\in P(\Qp)$ such that $gx\in M(B_{\dR})G(B_{\dR}^+)/G(B_{\dR}^+)$.
\end{proof}

\subsection{Opposite reductions}\label{secoppfilt}

Assume that $G$ is quasi-split. In this section we consider parabolic reductions whose slope vector is anti-dominant, and such that the associated $M$-bundle is semistable, to prove that Theorem \ref{thmmain1} implies Theorem \ref{thmoppfilt}. The idea for this proof was pointed out to us by D.~Hansen.

We consider the local charts for $\Bun_G$ introduced in \cite[V.3]{FS}. Let $[b]\in B(G)$. Let $M$ be the centralizer of its Newton point $\nu_b$ and let $P$ be the parabolic subgroup with Levi factor $M$ such that $\nu_b$ is $P$-anti-dominant. Then $\mathcal M_b$ is defined as the v-stack assigning to any perfectoid space $S$ over $\overline{\mathbb F}_q$ the groupoid of $P$-bundles $\E_P$ over $S$ such that $\E_P\times^P M$ is the $M$-bundle $\E_{b_M}^M$ associated with the reduction of $[b]$ to the centralizer of its Newton point. It is a cohomologically smooth Artin v-stack.

We consider the natural map $\pi_b:\mathcal M_b\rightarrow \Bun_G$ mapping $\E_P$ to $\E_P\times^P G$. By \cite[Thm.~V.3.7]{FS}, it is partially proper, representable in locally spatial diamonds and cohomologically smooth.

We can now prove that Theorem \ref{thmmain1} implies Theorem \ref{thmoppfilt}.

\begin{proof}[Proof of Theorem \ref{thmoppfilt}]
Let $[b]\geq [b']$ in $B(G).$ We have to show that the point of $\Bun_G$ corresponding to $[b']$ is in the image of $\pi_b$. However, since $\pi_b$ is cohomologically smooth, the induced map on topological spaces is open. The split $P$-bundle $\E_{b_M}^M\times^M P$ is mapped to $[b]$. Hence by Theorem \ref{thmmain1}, the image of $\pi_b$ contains all $[b']\leq [b]$.
\end{proof}

\section{The weakly admissible locus}\label{secwa}
\begin{definition}
\begin{enumerate}
\item A point $x\in \Gr_{G}(C)$ is weakly admissible for $b=1$ if and only if for any parabolic subgroup $P$ of $G$, the reduction $(\E_{1,x})_P$ of $\E_{1,x}$ induced by the reduction $\E_{1}^P$ of $\E_1$ is non-positive. 
\item Let $b\in G(\breve{\Qp})$ be basic. Then $x\in \Gr_{G}(C)$ is weakly admissible for $b$ if the corresponding point $x\in \Gr_{G_b}(C)$ as in Section \ref{remcompgj} is weakly admissible for $1$.
\item We denote the set of weakly admissible points (for $b$) by $\Gr_{G,b}(C)^{\wa}$.
\end{enumerate}
\end{definition}
From this we immediately obtain the following lemma that in particular allows to reduce the computation of weakly admissible points from a reductive group to the quasi-split inner form of its adjoint group.
\begin{lemma}\label{corwatoqsp}
Let $b\in G(\breve{\Qp})$ be basic.
\begin{enumerate}
\item $x\in \Gr_G(C)$ is weakly admissible for $b$ if and only if $x_{\ad}\in \Gr_{G_{\ad}}(C)$ is weakly admissible for $b_{\ad}$.
\item Let $b_0\in G(\breve{\Qp})$ be basic. Then $x\in \Gr_G(C)$ is weakly admissible for $b$ if and only if $x\in \Gr_{G_{b_0}}(C)$ is weakly admissible for $bb_0^{-1}$.
\end{enumerate}
\end{lemma}

\begin{lemma}\label{lemqspwa}
Assume that $G$ is quasi-split and let $b\in G(\breve{\Qp})$ be basic. Then $x\in \Gr_{G}(C)$ is weakly admissible for $b$ if and only if for every standard parabolic subgroup $P$ with standard Levi factor $M$ and every reduction $b_M=g^{-1}b\sigma(g)$ of $b$ to $M$, the reduction $(\E_{b,x})_P$ of $\E_{b,x}$ induced by the reduction $\E_{b_M}^M\times^M P$ of $\E_b$ is non-positive.
\end{lemma}

\begin{proof}
Let $P$ be a parabolic subgroup of $G_b$. Then $P$ can be seen as a parabolic subgroup of $G_{\breve{\Qp}}$ stable under $b\sigma$. Since $G$ is quasi-split, $P$ is conjugate (by some $g$) to a standard parabolic subgroup $P'$ of $G$, which is then stable under $g^{-1}b\sigma(g)\sigma$. Since it is defined over ${\Qp}$, this holds if and only if $g^{-1}b\sigma(g)$ is in the stabilizer of $P'$, which equals $P'$. Modifying $g$ by a suitable element of the unipotent radical of $P'$ we may assume that $g^{-1}b\sigma(g)$ is in the standard Levi factor of $P'$. In this way $\sigma$-conjugation with $g$ translates between reductions of $\E_{b,x}$ as in the lemma, and reductions of $\E_{1,x}^{G_b}$ to $P$.
\end{proof}

\begin{remark}
Using the preceeding two lemmas, one could also define the weakly admissible locus for non-basic $b$: An $x\in \Gr_{G}(C)$ is weakly admissible for $b$ if the corresponding element $x_{\ad}\in\Gr_{H}(C)$ (where $H$ is the quasi-split inner form of $G_{\ad}$) satisfies the property of Lemma \ref{lemqspwa}.
\end{remark}

\begin{remark}\label{remsb}
Let $G$ be quasi-split. Recall that a $\sigma$-conjugacy class $[b]\in B(G)$ is superbasic if it does not have a reduction to any proper Levi subgroup of $G$. Every $[b]\in B(G)$ has a reduction $b_M$ to a standard Levi subgroup $M$ such that $[b_M]_M$ is superbasic in $M$ and such that the $M$-dominant Newton point of $[b_M]_M$ is $G$-dominant. Since $M$ is standard and corresponds to a minimal Levi subgroup of $G_b$, it is uniquely defined by $[b]$.
\end{remark}

For any standard Levi subgroup $M$ we consider the averaging map
$$\av_M:X_*(A)_{\Q}\rightarrow \pi_1(M)_{\Gamma,\Q}\overset{\sim}{\longrightarrow} X_*(Z_M)_{\Q}^{\Gamma}\rightarrow X_*(A)_{\Q}$$
as well as the corresponding map $$\av_M:X_*(T)_{\Q}\rightarrow \pi_1(M)_{\Gamma,\Q}\overset{\sim}{\longrightarrow} X_*(Z_M)_{\Q}^{\Gamma}\rightarrow X_*(A)_{\Q}.$$

\begin{lemma}\label{lemcompwa}
Let $G$ be a quasi-split connected reductive group over ${\Qp}$ and let $b\in G(\breve{\Qp})$ be basic. Let $P$ be a standard parabolic subgroup of $G$ with standard Levi factor $M$ and such that $b$ has a reduction $b_M=g^{-1}b\sigma(g)$ to $M$ with $[b_M]_M$ superbasic in $M$. Then $x\in \Gr_{G}(C)$ is weakly admissible if and only for every $j\in G_b({\Qp})$ and corresponding $\lambda$ with $jx\in (gUg^{-1})(B_{\dR})\lambda(\xi)G(B_{\dR}^+)/G(B_{\dR}^+)$ we have $\av_M(-\lambda)\leq \av_G(-\lambda)$.
\end{lemma}

\begin{proof}
Replacing $b$ by $b_M$ replaces $G_b$ by $G_{b_M}=g^{-1}G_bg$, and $x$ is weakly admissible for $b$ if and only if $g^{-1}x$ is weakly admissible for $b_M$. Thus we may assume that $b=b_M$ and $g=1$.

By Lemma \ref{lemqspwa}, $x$ is weakly admissible if and only if for every standard parabolic subgroup $P'$ with standard Levi factor $M'$ and every reduction $b_{M'}$ of $b$ to $M'$, the reduction $(\E_{b,x})_{P'}$ of $\E_{b,x}$ induced by the reduction $\E_{b_{M'}}^{M'}\times^{M'} P'$ of $\E_b$ is non-positive. By Remark \ref{remsb}, $M$ is the unique minimal element among the standard Levi subgroups containing a reduction of $b$. Hence $M\subseteq M'$, and $b_{M'}=h^{-1}b\sigma(h)$ with $h=jh'\in G_{b}({\Qp})M'(\breve{\Qp})$. Thus the reduction $(\E_{b,x})_{P'}$ is a coarsening of the reduction of $\E_{b,x}$ to $P$ obtained from the reduction $j^{-1}b\sigma(j)=b$ of $b$ to $M$. If this finer reduction is non-positive, then the same holds for $(\E_{b,x})_{P'}$. Thus $x$ is weakly admissible if and only if for every reduction $j^{-1}b\sigma(j)=b$ of $b$ to $M$, the corresponding reduction $(\E_{b,x})_{P}$ of $\E_{b,x}$ is non-positive. Now we use Lemma \ref{lemslopekappa} and the fact that $\nu_{b}$ is central to compute the slope vectors of these reductions. The reductions are non-positive if and only if for each $j$, the corresponding $\lambda\in X_*(T)$ with $jx\in U(B_{\dR})\lambda(\xi)G(B_{\dR}^+)/G(B_{\dR}^+)$ satisfies $\av_M(-\lambda)\leq \av_G(-\lambda)$.
\end{proof}

The next lemma is the generalization to our context of the assertion that admissible implies weakly admissible.
\begin{lemma}\label{lem47}
Let $x\in \Gr_{G}(C)$ be such that $\E_{b,x}$ is semi-stable. Then $x$ is weakly admissible for $b$.
\end{lemma}
\begin{proof}
By Lemma \ref{corwatoqsp} and Section \ref{secadinn} we may assume that $G$ is quasi-split, replacing it by a quasi-split inner form of its adjoint group. The condition that $\E_{b,x}$ is semi-stable implies that for every parabolic subgroup $P$ of $G$, every reduction of $\E_{b,x}$ to $P$ is non-positive. Weak admissibility requires this condition only for particular such reductions.
\end{proof}

By Corollary \ref{remnonempty} below, an $x$ as in the lemma exists at least if $b$ itself is basic. Altogether we obtain
\begin{prop}
Let $G$ be a reductive group over ${\Qp}$, let $\{\mu\}$ be a conjugacy class of cocharacters of $G$ and $E$ its field of definition. Let $b\in G(\breve{\Qp})$ be basic. Consider the subfunctor $\Gr_{G,\mu,b}^{\wa}$ of $\Gr_{G,\mu}$ with $\Gr_{G,\mu,b}^{\wa}(S)$ consisting of those elements of $\Gr_{G,\mu}(S)$ such that for every geometric point of $S$, the associated element of $\Gr_{G,\mu}(C)$ lies in $\Gr_{G,\mu,b}(C)^{\wa}$. Then $\Gr_{G,\mu,b}^{\wa}$ defines a locally spatial diamond over $E$ which is open in $\Gr_{G,\mu}$.
\end{prop}

\begin{proof}
As usual we may assume that $G$ is quasi-split. We use the notation of Lemma \ref{lemcompwa}. We may assume that $b=b_M$. By Lemma \ref{lemcompwa}, the complement in $\Gr_{G,\mu}$ is a union of translates (by elements of $G_b(\Qp)$) of the subspaces $S_{\lambda}\cap \Gr_{G,\mu}$ where $\lambda_{\dom}\leq (-\mu)_{\dom}$ and with $\av_M(-\lambda)\nleq \av_G(-\lambda)$. The union of the above subspaces $S_{\lambda}\cap \Gr_{G,\mu}$ is stable under $P_b(\Qp)$. In particular, the above union of translates of this union of subspaces is profinite.

Furthermore, $\av_M(-\lambda)\nleq \av_G(-\lambda)$ implies the same condition for every $\lambda'\leq \lambda$. From the closure relations for the $S_{\lambda}$ as in \cite[Prop.~6.4]{Shen} we thus obtain that the complement is closed in $\Gr_{G,\mu}$, which implies that $\Gr_{G,\mu,b}^{\wa}$ is a locally spatial diamond over $E$, and open in $\Gr_{G,\mu}$.
\end{proof}
\begin{remark}
From the density of the basic Newton stratum that we prove in Corollary \ref{corclosns} and Lemma \ref{lem47} we obtain that the weakly admissible locus is also dense in $\Gr_{G,\mu}$.
\end{remark}

In the remainder of this section we compare our notion of weak admissibility to the semi-stable locus in flag varieties via the Bialynicki-Birula map, compare Remark \ref{remBB}.

Let $G$ be a reductive group over ${\Qp}$ and let $\{\mu\}$ be a conjugacy class of cocharacters of $G$. Let $E$ be its field of definition. We denote by $\F(G,\mu)$ the associated flag variety over $E$. Let $b\in G(\breve{\Qp})$. Let $C$ be a complete field extension of $\breve{\Qp}$. For $x\in \F(G,\mu)(C)$ we denote by $\mu_x\in \{\mu\}$ the associated cocharacter, which is then defined over $C$. For every representation $(V,\rho)$ of $G$ we have an associated filtered isocrystal defined as $(V_{\breve{\Qp}}, \rho(b)\sigma, {\rm Fil}^{\bullet}_{\rho\circ \mu_x}V_C)$. Then $x$ is semi-stable if for every $(V,\rho)$, the associated filtered isocrystal is semi-stable, compare \cite[Def.~8.1.5]{DOR} for details. This defines a partially proper open adic subspace $\F(G,\mu,b)^{\sst}\subseteq \F(G,\mu)$ whose $C$-valued points are the semi-stable points defined above. If $\kappa_G(b)=\mu^{\sharp}\in \pi_1(G)_{\Gamma,\Q}$, such $x$ are also called weakly admissible.

\begin{lemma}
Let $\{\mu\}$ be minuscule. Then the Bialynicki-Birula map induces an isomorphism $\Gr_{G,\mu,b}^{ \wa}\rightarrow \F(G,\mu,b)^{{\rm ss},\diamond}$.
\end{lemma}
\begin{proof}
The Bialynicki-Birula map is an isomorphism between the affine Schubert cell and the flag variety. Thus it remains to show that the claimed restriction exists and is a surjection on $C$-points for any algebraically closed complete extension of ${\Qp}$.

Let $x\in \Gr_{G,\mu}(C)$. It is weakly admissible for $b$ if and only if $x_{\ad}\in\Gr_{G_{\ad},\mu_{\ad}}(C)$ is weakly admissible for $b_{\ad}$. By \cite[Prop.~9.5.3(iv)]{DOR} an analogous statement holds for the semi-stable loci. Thus we may assume that $G$ is adjoint.

Let $G'$ be a quasi-split inner form of $G$, obtained by twisting with a basic element $b_0\in G(\breve{\Qp})$. Let $b\in G(\breve{\Qp})$ and let $\tilde b=bb_0^{-1}\in G'(\breve{\Qp})$ the corresponding element. Again by \cite[Prop. 9.5.3]{DOR}, we obtain an identification of the associated semi-stable loci $$\F(G,\mu,b)^{\sst}=\F(G',\mu,\tilde b)^{\sst}.$$ By Lemma \ref{corwatoqsp} above we have a corresponding comparison for the weakly admissible loci in the affine Schubert cell. Hence it is enough to prove the lemma for quasi-split $G$, in which case it is an immediate consequence of \cite[Prop. 2.7]{CFS} together with Lemma \ref{lemqspwa} above.
\end{proof}

\begin{ex}
For non-minuscule $\mu$, the Bialynicki-Birula map is not an isomorphism and the notions of weakly admissible resp.~semi-stable loci do not correspond to each other any more, as we now illustrate. In particular, our definition of weak admissibility coincides with the corresponding definition in the latest version of \cite{Shen}, but not with the definition in previous versions of Shen's article.

We call $x\in \Gr_{G,\mu}(C)$ classically weakly admissible if ${\rm BB}(x)$ is semi-stable in the above sense.

We consider the case $G=\GL_{2,\mathbb Q_p}$, $b={\rm diag}(p^2,p^2)$ of constant Newton slope 2, and $\mu=(4,0)$. Notice that this is even an example where $\kappa_G(b)=\mu^{\sharp}$, so that the flag variety notions of weak admissibility and of semi-stability coincide. Furthermore, $G_b=G$ in this case.

We first compute the classically weakly admissible points. The flag variety $\F(G,\mu)$ is decomposed into two Schubert cells, the class of the identity element forms the Schubert cell for $1$ whereas its open complement is the Schubert cell for $s=(12)$. The semistable locus is the complement of the $G(\mathbb Q_p)$-orbit of the closed Schubert cell.

Let $I\subset G(B_{\dR}^+)$ be the subgroup of elements whose image in $G(C)$ is in the Borel subgroup $B$ of upper triangular matrices. Thus an element $x\in \Gr_{G,\mu}(C)$ is classically weakly admissible if and only if for all $j\in G_b({\mathbb Q_p})=G({\mathbb Q_p})$ we have
\begin{align*}
jx&\in Is\mu^{-1}(\xi)G(B_{\dR}^+)/G(B_{\dR}^+)\\
&=I(\mu^{-1})_{\dom}(\xi)G(B_{\dR}^+)/G(B_{\dR}^+)\\
&=U(B_{\dR}^+)(\mu^{-1})_{\dom}(\xi)G(B_{\dR}^+)/G(B_{\dR}^+)
\end{align*}
 where $U$ is the unipotent radical of $B$. If this is the case, then in particular, $jx\in S_{(\mu^{-1})_{\dom}}(C)$ for all $j$, which by Lemma \ref{lemcompwa} implies that $x$ is weakly admissible.

Let $\lambda=(-1,-3)$. The points of $(S_{\lambda}\cap \Gr_{G,\mu})(C)\neq\emptyset$ do not satisfy the above condition, and are thus not classically weakly admissible.

Consider now the complement of the weakly admissible locus. It is a profinite union of  $G({\mathbb Q_p})$-translates of the intersections of $\Gr_{G,\mu}$ with $S_{\lambda'}$ for $\lambda'\in \{(-3,-1),(-4,0)\}$. These intersections are of dimensions 1 and 0, respectively, as can be shown by an explicit calculation. In particular, they are of dimension less than that of $S_{\lambda}\cap \Gr_{G,\mu}$, which is 3. Hence there are weakly admissible points (in $S_{\lambda}\cap \Gr_{G,\mu}(C)$) which are not classically weakly admissible.

The function field analog of this example has been studied (using different methods) by Hartl, \cite[Ex.~3.3.2, $d=4$]{HartlAnn}.
\end{ex}

\section{Classical points}\label{secclpt}

In this section we consider the particular properties of points defined over a finite extension $K$ of $\breve{\Qp}$. These generalize well-known results (such as for example the theorem of Colmez and Fontaine that weakly admissible implies admissible) to our setting. As an application we determine which Newton strata have classical points.

Since we are not interested in results for a particular such field $K$, but rather in the set of all classical points, we assume throughout that $K$ is a sufficiently large finite extension of $\breve{\Qp}$ so that all relevant elements, subgroups, etc. are defined over $K$. We choose a split maximal torus of $G$ defined over $K$, and contained in some Borel subgroup $B$ whose unipotent radical we denote by $U$. Let $\{\mu\}$ be a conjugacy class of cocharacters of $G$.

\begin{prop}\label{propclBB}
The Bialynicki-Birula map induces a bijection $$\Gr_{G,\mu}(K)\rightarrow \F(G,\mu)(K).$$
\end{prop}

\begin{proof}
For $G=\GL_n$, this is shown in \cite[Prop.~10.4.4]{FF}. The Bialynicki-Birula map is functorial. Choosing a faithful representation of $G$, we thus obtain that also the Bialynicki-Birula map for $G$ is injective on $K$-valued points. To show that it is surjective, let $g\in \F(G,\mu)(K)$. Let $\tilde g\in G(K)$ be an inverse image under the projection $G\rightarrow \F(G,\mu)$. Using that $K\hookrightarrow B^{+}_{\dR}(K)$, we obtain an element $x=\tilde g\mu^{-1}(\xi)G(B_{\dR}^+(K))/G(B_{\dR}^+(K))\in \Gr_{G,\mu}(K)$ with ${\rm BB}(x)=g$.
\end{proof}

We now fix a basic element $b\in G(\breve{\Qp})$.
\begin{thm}\label{propwaa}
Let $x\in \Gr_{G,\mu}(K)$. Then the following are equivalent.
\begin{enumerate}
\item $x\in \Gr_{G,\mu,b}^{\a}(K),$
\item $x\in \Gr_{G,\mu,b}^{\wa}(K),$ 
\item ${\rm BB}(x)\in \F(G,\mu,b)^{{\rm ss}}(K).$
\end{enumerate}
If $\kappa_G(b)=\mu^{\sharp}$, this is also equivalent to ${\rm BB}(x)\in \F(G,\mu,b)^{\a}(K),$ and by definition we then have $\F(G,\mu,b)^{{\rm ss}}(K)=\F(G,\mu,b)^{\wa}(K)$.
\end{thm}

\begin{lemma}\label{lemclxy}
Assume that $G$ is quasi-split. Let $x\in \Gr_{G,\mu}(K)$ and let $w\in W$ be such that $y={\rm BB}(x)$ is in the Schubert cell for $w$. Then $x\in S_{-\mu^w}(K)$.

In particular, an intersection $S_{\lambda}\cap \Gr_{G,\mu}$ has classical points if and only if $-\lambda\in W.\mu$.
\end{lemma}
\begin{proof}
We have $y\in U(K)wP_{\mu}(K)/P_{\mu}(K)$. Let $\tilde y$ be a representative in $U(K)w$, and $\tilde x$ the image of $\tilde y \mu^{-1}(\xi)$ in $\Gr_{G}(K)$ (using that $K$ is a subring of $B_{\dR}^+(K)$). Then by definition $\tilde x\in S_{-\mu^{w}}(K)$, and $\tilde x$ is the unique preimage of $y$ under the Bialynicki-Birula map. Thus $x=\tilde x$ is as claimed.
\end{proof}

\begin{proof}[Proof of Theorem \ref{propwaa}]
As usual we may assume that $G$ is adjoint and quasi-split over ${\Qp}$, and we choose a standard parabolic subgroup $P$ of $G$ with standard Levi subgroup $M$ and such that $b$ has a reduction $b_M$ to $M$ that is superbasic in $M$. Replacing $b$ by $b_M$ modifies all relevant subspaces by left multiplication by a fixed element $g\in G(\breve{\Qp})$. Thus we assume that $b=b_M$ from now on.

We begin by proving the equivalence of (2) and (3). The complement of the semi-stable locus is a profinite union of Schubert cells. More precisely, $y:={\rm BB}(x)$ is semi-stable if and only if for every $j\in G_b({\Qp})$, the element $jy$ is in a Schubert cell for some $w$ with $\av_M(\mu^w)\leq \av_G(\mu^w)$. By Lemma \ref{lemclxy} this is in turn equivalent to $jx\in S_{-\mu^w}(K)$ for the same $w$. But this is by definition the same as the condition that $x$ is weakly admissible (the slope vector of the reduction of $\E_{b,jx}$ to $P$ being $-\av_M(-\mu^w)=\av_M(\mu^w)$).

Now we prove the equivalence of (1) and (3). For $G=\GL_n$, this is \cite[Prop.~10.5.6]{FF}. For general (adjoint) $G$ we choose a faithful representation $(\rho,V)$ of $G$ which is then also homogenous in the sense of \cite[5.1]{DOR}. Then an element $y\in \F(G,\mu)(K)$ is semi-stable if and only if $\rho(x)\in\F(\GL(V),\rho\circ \mu)$ is semi-stable (by \cite[Prop.~9.5.3]{DOR}). By the result for $\GL_n$, this is equivalent to $\rho(\BB_G^{-1}(y))=\BB_{\GL(V)}^{-1}(\rho(y))$ being admissible, where $\BB^{-1}$ denotes the inverse of the Bialynicki-Birula bijection of Proposition \ref{propclBB}. It remains to show that a classical point $x\in \Gr_{G,\mu}(K)$ is admissible for $b$ if and only if $\rho(x)$ is admissible for $\rho(b)$. Let $\E^G_{b,x}\cong \E^G_{b'}$ for some $[b']\in B(G)$. Then $\E^{\GL(V)}_{\rho(b),\rho(x)}\cong \E^{\GL(V)}_{\rho(b')}$. The element $x$ is admissible if and only if $\nu_{b'}$ is central. A central element of $\mathcal{N}(G)$ is determined by its image in $\pi_1(G)_{\mathbb Q}$. Since $G$ is adjoint, $\nu_{b'}$ is thus central if and only if it is trivial. This is equivalent to $\rho(\nu_{b'})=\nu_{\rho(b')}=1$.

The last assertion is well-known. To prove it, one reduces again to $\GL_n$, in which case this is due to Colmez and Fontaine \cite{CF}.
\end{proof}

\begin{cor}\label{remnonempty}
$\Gr_{G,\mu,b}^{[b']}$ is non-empty if and only if $[b']\in \BGmub$. Furthermore, the Newton stratum  $\Gr_{G,\mu,b}^{[b'_0]}$ for the basic class $[b'_0]\in \BGmub$ has a classical point.
\end{cor}
\begin{proof}
For minuscule $\mu$, this is shown by Rapoport in \cite[Cor.~A.10]{R}. The proof of his statement uses as an essential step the minuscule case of Theorem \ref{propwaa} above. Using Theorem \ref{propwaa} in general, Rapoport's argument then carries over almost literally to the non-minuscule case (the necessary results in \cite{DOR} making no assumption on $\mu$).
\end{proof}

\begin{thm}\label{propclpt}
Let $G$ be a connected reductive group over ${\Qp}$ and $\{\mu\}$ a conjugacy class of cocharacters of $G$. By $\mu$ we denote as usual the dominant representative in $X_*(T)$. Let $b\in G(\breve{\Qp})$ be basic.
\begin{enumerate}
\item Let $[b']\in \BGmub$. Then $\Gr_{G,\mu,b}^{[b']}$ has a classical point if and only if there is a $w\in W$ such that $\nu_{b'}^{\sharp_M}=\nu_b^{\sharp_M}-\mu^{w,\sharp_M}$ where $M$ is the centralizer of the Newton point of $[b']$.
\item Let $[b']\in \BGmub$. Then $\Gr_{G,\leq\mu,b}^{[b']}$ has a classical point if and only if $G_b$ has a parabolic subgroup $P$ such that $[b'b^{-1}]_{G_b}$ has a representative in $P$ whose image in the Levi quotient $M$ is basic in $M$.
\item If $G$ is quasi-split, the condition in (2) is equivalent to the condition that $b$ has a reduction to the centralizer of $\nu_{b'}$.
\end{enumerate}
\end{thm}

\begin{remark}
In particular, a non-empty Newton stratum $\Gr_{G,\mu}^{[b']}$ for minuscule $\mu$ and $[b']\in \BGmub$ has a classical point if and only if the condition in (2) is satisfied.

Furthermore, if $[b']\in B(G)$, then $\Gr_{G}$ has a classical point $x$ with $\E_{b,x}\cong \E_{b'}$ if and only if the condition in (2) is satisfied.
\end{remark}
\begin{ex}
Assume that $G$ is quasi-split and that $[b]\in B(G)$ is superbasic, i.e.~no $\sigma$-conjugate of $b$ is contained in a proper Levi subgroup of $G$. Then all classical points of $\Gr_{G,\mu,b}$ are in the basic Newton stratum.
\end{ex}

\begin{proof}
Since we always assume that $\kappa_G(b')=\kappa_G(b)-\mu^{\sharp}$, each of the above conditions is equivalent to the respective condition for the images in $G_{\ad}$. Thus we may assume that $G$ is adjoint. Furthermore, replacing $G$ by an inner twist by some $b_0$ we may assume also in (1) and (2) that $G$ is quasi-split.

If $[b']$ is basic, the Newton stratum has a classical point by Corollary \ref{remnonempty}. Also, all of the other conditions are satisfied. Thus the theorem holds in this case. From now on we assume that $[b']$ is not basic.

In the context of (1) let $x\in \Gr_{G,\mu,b}^{[b']}(K)$. Since $x$ is not admissible, it is also not weakly admissible. Thus there is a standard parabolic subgroup $P$ together with a reduction $b_M$ of $b$ to its Levi factor $M$ such that the reduction of $\E_{b,x}$ to $P$ induced by $\E_{b_M}^M\times^M P$ has a slope vector $v$ with $v\nleq \av_G(v)$. We choose $P$ and the reduction of $b$ in such a way that $v$ becomes maximal (i.e., there is no strictly bigger $v'$ for any such reduction of $\E_{b,x}$).

{\it Claim.} $v$ is dominant.

Assume that $v$ is not dominant. Since it is central in $M$, there is a simple root $\alpha$ in the unipotent radical of $P$ such that $\langle \alpha,v\rangle<0$. Let $P'\supsetneq P$ be the parabolic subgroup corresponding to the simple roots in $M$ together with $\alpha$. Then $v$ is anti-dominant in $M'$. Consider the reduction of $\E_{b,x}$ to $P'$ for the same reduction $b_M$ of $b$. Its slope vector is central in $M'$ with the same image in $\pi_1(M')_{\Q,\Gamma}$. In particular, it is strictly bigger than $v$, contradiction.

By maximality of $v$, the modification $(\E_{b,x})_P\times^P M$ of $\E^M_{b_M}$ is weakly admissible (compare the generalities in the proof of \cite[Lemma 6.4]{CFS}). Since $x$ is classical, also $\pr_M(x)$ is classical. Hence this modification of $M$-bundles is admissible. Thus $(\E_{b,x})_P$ is a $P$-bundle with dominant slope vector $v$ and such that the associated $M$-bundle is semi-stable. As in the proof of \cite[Prop. 5.16]{FGtorseurs}, this implies that $(\E_{b,x})_P\times^P M$ is a reduction of $(\E_{b,x})_P$ to $M$, and hence a reduction of $\E_{b,x}$ to $M$. We obtain $\nu_{b'}=v^*$.

Let $g\in G(\breve{\Qp})$ with $b_M=g^{-1}b\sigma(g)$. Replacing $x$ by $g^{-1}x$ (which is still classical) we may assume that $b=b_M$ and that the reduction of $b=b_M$ to $b_M$ is the identity. Let $\lambda\in X_*(T)$ be such that $x\in S_{\lambda}(C)$. By Lemma \ref{lemclxy}, $-\lambda=\mu^w$ for some $w\in W$. We also have a reduction of $b'$ to the centralizer of its Newton point, which we denote again by $b'$. Then by Lemma \ref{lemslopekappa}, we obtain $\nu_{b'}^{\sharp_M}=\kappa_M(b')=\kappa_M(b_M)+w_0(\lambda)^{\sharp_M}=\nu_b^{\sharp_M}-\mu^{w_0w,\sharp_M}$.

For the converse assume that there is a $w\in W$ such that $\nu_{b'}^{\sharp_M}=\nu_b^{\sharp_M}-\mu^{w,\sharp_M}$ where $M$ is the centralizer of the Newton point of $[b']$. We first show that $b$ has a reduction to $M$. Again we may assume that $b'\in M(\breve{\Qp})$. Let $\tilde b_M\in M(\breve{\Qp})$ be basic in $M$ with $\kappa_M(\tilde b_M)=\kappa_M(b')+\mu^{w,\sharp_M}$. Then the $M$-dominant Newton point $\nu_{\tilde b_M}$ is the unique element that is central in $M$ and with $\nu_{\tilde b_M}^{\sharp_M}=\nu_{b'}^{\sharp_M}+\mu^{w,\sharp_M}$. Thus by our assumption, $\tilde b_M$ is the desired reduction of $b$ to $M$. By Corollary \ref{remnonempty}, the basic locus for modifications of $\E_{\tilde b_M}^M$ in $\Gr_{M,{\mu^w_{M-\dom}}}\subseteq \Gr_{G,\mu}$ has classical points, and the corresponding basic class has Kottwitz point $\kappa_M(\tilde b_M)-\mu^{w,\sharp_M}=\kappa_M(b')$. Hence it is equal to $[b']_M$, which finishes the proof of (1).

(3) holds since parabolic subgroups of $G$ containing $b$ (and thus stable under $b\sigma$) are in bijection with parabolic subgroups of $G_b$. Thus for (2), it remains to show that for quasi-split groups, the condition in (3) is equivalent to the existence of a classical point in $\Gr_{G,\leq\mu,b}^{[b']}$.

If $\Gr_{G,\leq\mu,b}^{[b']}$ has a classical point, then the same holds for $\Gr_{G,\mu',b}^{[b']}$ for some $\mu'\leq \mu$. From the proof of (1) we see that this implies that $[b]$ has a reduction to $M$, which proves one direction. Conversely, assume that $b_M$ is contained in the centralizer of $\nu_{b'}$, and that $[b']\in \BGmub$. Using (1) it is enough to show that there is a $\mu'\leq\mu$ and a $w\in W$ with $\kappa_M(b')=\kappa_M(b_M)-(\mu')^{w,\sharp_M}$. For the proof of this assertion we pass to the inner form $G_b$ of $G$, and may thus assume that $b=1$. Then $G$ is in general no longer quasi-split, but by our asumption on $b$, the centralizer $M$ of the Newton point $\nu_{b'}$ and the associated parabolic subgroup $P$ are still defined over $F$. We may assume that $b'\in M(\breve F)$. We now have $[b']\in B(G,-\mu)$. By \cite[Thm.~A]{HeNonempty}, this implies that there is a $g\in G(B_{\dR})$ with $g^{-1}b'\sigma(g)\in G(B_{\dR}^+)\mu^{-1}(\xi)G(B_{\dR}^+)$.
We use the Iwasawa decomposition to write $g=mnk$ with $m\in M(B_{\dR})$, $n\in N(B_{\dR})$ where $N$ is the unipotent radical of $P$, and $k\in G(B_{\dR}^+)$. We may assume that $k=1$. Then $g^{-1}b'\sigma(g)\in P(B_{\dR})\cap \Gr_{G,\mu}(C)$. Let $\tilde b=m^{-1}b'\sigma(m)$. We have $g^{-1}b'\sigma(g)=n^{-1}m^{-1}b'\sigma(mn)=n^{-1}[\tilde b\sigma(n)\tilde b^{-1}]\tilde b$ with $n^{-1}[\tilde b\sigma(n)\tilde b^{-1}]\in N(B_{\dR})$. Thus $\tilde b=m^{-1}b'\sigma(m)$ is the retraction of $g^{-1}b'\sigma(g)$ to $M$, and as such lies in $\Gr_{G,\leq\mu}$. Let $\mu_0\in X_*(T)$ be such that $\tilde b\in \Gr_{M,\mu_0}(C)$. Since $\tilde b\in [b']_M$, this implies $[b']_M\in B(M,-\mu_0)$. Furthermore, $\mu_{0,\dom}\leq \mu$, and $(-\mu_0)^{\sharp_M}=\kappa_M(\tilde b)=\kappa_M(b')$. Thus $\mu_0$ has all required properties of $(\mu')^{w}$.
\end{proof}
\begin{remark}
From the first part of the proof we obtain the interesting fact that for classical points, the Newton point of $\E_{b,x}$ is $v^*$ where $v$ is a maximal element in the set of slope vectors of reductions of $\E_{b,x}$ to parabolic subgroups  $P$ induced by $\E_{b_M}^M\times^M P$ for some reduction $b_M$ of $b$ to $M$. Such maximal slope vectors $v$ are studied more systematically in subsequent joint work with K. H. Nguyen \cite{hnstrata}.
\end{remark}

\section{The topology of $\Bun_G$}
\subsection{The generic Newton stratum in a semi-infinite cell}

\begin{lemma}
Assume that $G$ is quasi-split and let $\lambda\in X_*(T)$. Then the set  $$B(G,\geq w_0(\lambda)):=\{[b']\in B(G)\mid \kappa_G(b')=\lambda^{\sharp},w_0( \lambda)^{\diamond}\leq\nu_{b'}\}$$  has a unique minimal element that we denote $[b(\lambda)]$. We have $[b(\lambda)]\leq[\lambda(\xi)]$.
\end{lemma}

{\setlength{\unitlength}{16 pt}
\begin{figure}[h]
\begin{center}
\begin{picture}(8,2.7)(0,0.5)
\multiput(0,0)(1,0){8}{$\cdot$}
\multiput(0,1)(1,0){8}{$\cdot$}
\multiput(0,2)(1,0){8}{$\cdot$}
\multiput(0,3)(1,0){8}{$\cdot$}

\put(0.1,0.17){\line(1,0){1}}
\put(1.1,0.17){\line(1,1){1}}
\put(2.1,1.17){\line(1,0){1}}
\put(3.1,1.17){\line(1,1){1}}
\put(4.1,2.17){\line(1,0){2}}
\put(6.1,2.17){\line(1,1){1}}

\put(0.1,0.17){\line(2,1){4}}
\put(4.1,2.17){\line(3,1){3}}
\put(5.2,3.3){$\nu_{b(\lambda)}$}
\put(3.3,0.6){$w_0(\lambda)$}
\end{picture}
\end{center}
\caption{An example of $w_0(\lambda)$ and $\nu_{b(\lambda)}$ for $\GL_7$} \label{fig61}
\end{figure}
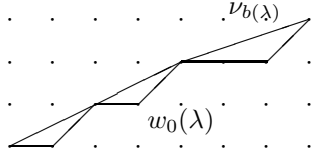
}

See Figure \ref{fig61} for an example of some $\nu_b$ and $w_0(\lambda)$ for the minimal $\lambda$ with $\nu_b=\nu_{b(\lambda)}$. Notice that in this example, both $[b]$ is minimal in $B(G,\geq w_0(\lambda))$ and $\lambda$ is minimal with $\nu_b=\nu_{b(\lambda)}$.

Note that the inequality in the definition of $B(G,\geq w_0(\lambda))$ compares the Newton point to a not necessarily dominant element, whereas the inequality $[b(\lambda)]\leq[\lambda(\xi)]$ includes a comparison between the associated dominant Newton points $\nu_{b(\lambda)}$ and $(\lambda^{\diamond})_{\dom}$.

For the proof we use the notion of edge subsets of \cite{Chai}. An edge subset is a subset $E$ of $\bigcup_i \mathbb Re_i^*\subset X_*(T)_{\mathbb R}$ where the $e_i^*$ are the fundamental coweights. Let $x\in C:=X_*(T)_{\mathbb R,\dom}$ and consider its image $(x_i)_i$ in $\bigoplus_i \mathbb R_{\geq 0}e_i^*$. Then we define $C_{\geq E}$ to be the set of $x$ as above with $x_i\geq \lambda$ whenever $\lambda e_i^*\in E$. We call $E$ reduced if $C_{\geq E}\subsetneq C_{\geq E'}$ for every $E'\subsetneq E$. Then by \cite[Thm.~6.5]{Chai}, $E$ has a join, that is a unique minimal element of $C_{\geq E}$.
\begin{proof}
Associated with $w_0(\lambda)$ we consider the edge subset consisting of all projections of $w_0(\lambda)$ to some $\mathbb R e_i^*$. Let $E$ be an associated reduced edge subset. Let $[b(\lambda)]\in C_{\geq E}$ be the join of $E$ as in \cite[Thm.~6.5]{Chai}. By \cite[Thm. 7.3.2]{Chai}, $[b(\lambda)]\in X_*(T)_{\mathbb R,\dom}$ corresponds to an element of $B(G)$, which proves the first assertion. The second assertion follows from $[\lambda(\xi)]\in B(G,\geq w_0(\lambda))$.
\end{proof}

\begin{lemma}\label{lemcharbl}
Let $G$ be quasi-split, let $[b]\in B(G)$, and let $M$ be the centralizer of its Newton point. Let $\lambda\in X_*(T)$. Let $b$ be a representative of $[b]$ in $M$ whose $M$-dominant Newton point is $G$-dominant.
\begin{enumerate}
\item Then $[b(\lambda)]=[b]$ if and only if $\kappa_M(b)=w_0(\lambda)^{\sharp_M}$ and $w_0(\lambda^{\diamond})\leq \nu_b$.
\item Let $M'\subseteq M$ such that $[b]_M$ has a reduction $b_{M'}$ to $M'$ that is superbasic in $M'$. If $\lambda$ is minimal with $[b(\lambda)]=[b]$ then $\kappa_{M'}(b_{M'})=w_0(\lambda)^{\sharp_{M'}}$.
\end{enumerate}
\end{lemma}
\begin{proof}
The first assertion is a reformulation of the last assertion of \cite[Thm.~6.5]{Chai}.

For (2) assume that $[b(\lambda)]=[b]$ and that $\lambda$ is minimal with this property. Assume the assertion does not hold. Then there is a maximal standard parabolic subgroup $P_0$ containing $M'$, with Levi factor $M_0$ such that $\kappa_{M_0}(b_{M'})>w_0(\lambda)^{\sharp_{M_0}}$. Let $\alpha$ be a simple absolute root that is not in $M'$. Then $\kappa_{M_0}(b_{M'})\geq (w_0(\lambda)+\alpha^{\vee})^{\sharp_{M_0}}$. This implies $(w_0( \lambda)+\alpha^{\vee})^{\diamond}\leq \nu_{b_{M'}}$, contradicting the minimality of $\lambda$.
\end{proof}
\begin{remark}\label{remhv21}
The elements $\lambda$ that are minimal with $[b(\lambda)]=[b]$ for a given $[b]$ are precisely the $w_0$-conjugates of any representative in $X_*(T)$ of the element $\lambda_G(b)\in X^*(\hat T)_{\Gamma}$ constructed in \cite[Lemma/Def.~2.1]{HamacherViehmann}. The element $\lambda_G(b)\in X^*(\hat T)_{\Gamma}$ is characterized uniquely by the property that $\lambda_G(b)^{\sharp_G} = \kappa_G(b)$ and that for every relative fundamental coweight $\omega_{\hat G,F}^\vee$ of $\hat G$, one has
  \begin{equation*} 
   \langle \lambda_G(b) - \nu_G(b), \omega_{\hat G,F}^\vee \rangle \in (-1,0].
  \end{equation*}
\end{remark}

\begin{lemma}\label{lemest}
Assume that $G$ is quasi-split. Let $[b']\in B(G)$ with $\E_{1,x}\cong \E_{b'}$ for some $x\in S_{\lambda}(C)$. Then $[b']\in B(G,\geq w_0(\lambda))$.
\end{lemma}
\begin{proof}
By construction we have $\kappa_G(b')=\kappa_G(1)+\lambda^{\sharp_G}$. Let $\E_{1}^{B}$ be the trivial $B$-bundle on $X$. It induces a reduction of $\E_{1,x}$ to $B$. By Lemma \ref{lemslopekappa}, the slope vector of this reduction agrees with $-\lambda^{\diamond}$. By the comparison theorem for the Harder-Narasimhan reduction we obtain that $\nu_{b'}\geq w_0(\lambda^{\diamond})$.
\end{proof}

Let $\lambda\in X_*(T)$. For each $\eta\in X_*(T)_{\dom}$ let $S_{\lambda,\eta}$ be as in Section \ref{secthebg}. Let $\mu$ be such that $S_{\lambda,\eta}\subseteq \Gr_{G,\leq \mu}$ and for $[b']\in B(G)$ let $S_{\lambda,\eta}^{[b']}=S_{\lambda,\eta}\cap \Gr_{G,\leq \mu,1}^{[b']}$, a locally spatial diamond. Since we use these only for modifications of the trivial bundle, we do not include $b=1$ in the notation.
\begin{thm}\label{propkey}
Let $G$ be quasi-split. We consider modifications of the trivial $G$-bundle $\E_1$. Let $[b']\in B(G)$ and let $\lambda\in X_*(T)$ be minimal with the property that $[b(\lambda)]=[b']$. Then for all sufficiently regular $\eta\in X_*(T)_{\dom}$, the Newton stratum $S_{\lambda,\eta}^{[b(\lambda)]}$ is open and dense in $S_{\lambda,\eta}$. In particular, $S_{\lambda}^{[b({\lambda})]}=\bigcup_{\eta}S_{\lambda,\eta}^{[b(\lambda)]}$ is also open and dense in $S_{\lambda}$.
\end{thm}

\begin{proof}
The last assertion follows from \eqref{eqSle}.

By Lemma \ref{lemest}, $[b(\lambda)]$ is less than or equal to all isomorphism classes of $G$-bundles corresponding to points of $S_{\lambda}$. 
Thus the semi-continuity theorem of Scholze-Weinstein \cite[Cor.~22.5.1]{SW} implies that the Newton stratum for $[b(\lambda)]$ in $S_{\lambda,\eta}$ is open for every $\eta$.

By Proposition \ref{propslambdaeta} it is enough to show that the complement of $S^{[b(\lambda)]}_{\lambda,\eta}$ in $S_{\lambda,\eta}$ has dimension strictly smaller than $\langle 2\rho,\eta\rangle.$ Let $P$ be a standard parabolic subgroup of $G$ with standard Levi factor $M$, and such that $[b(\lambda)]$ has a reduction $b_M$ to $M$ which is superbasic in $M$ and whose $M$-dominant Newton point is $w_0(\nu_{b(\lambda)})$. By Lemma \ref{lemcharbl} (applied to $w_0b_Mw_0^{-1}$) we have $\kappa_{M}(b_{M})=\lambda^{\sharp_{M}}$. Assume that $x\in S_{\lambda,\eta}(C)$ is such that $(\E_{1,x})_P\times^P M$ has $M$-dominant Newton point $\nu_{b_M}$, where $(\E_{1,x})_P$ is the reduction induced by the reduction $\E_{1}^P$ of $\E_1$. Then the Harder-Narasimhan polygon of this semi-stable $M$-bundle is $-\nu_{b_M}=\nu_{b(\lambda)}^*$ (central in $M$), and in particular $G$-dominant. Hence this $M$-bundle is a reduction of $(\E_{1,x})_P$ and also of $\E_{1,x}$, in other words, the reduction $(\E_{1,x})_P$ is split. In particular, $x$ is then in the Newton stratum for $[b(\lambda)]$. By Lemma \ref{lemkey}, $(\E_{1,x})_P\times^P M\cong \E_{1,\pr_M(x)}^M$.

\noindent{\it Claim.} The complement of the Newton stratum for $[b_M]_M$ in $$\lambda(\xi)\eta(\xi)(U\cap M)(B_{\dR}^+)\eta(\xi)^{-1}(U\cap M)(B_{\dR}^+)/(U\cap M)(B_{\dR}^+)=S^M_{\lambda,\eta}$$ has dimension strictly smaller than $\langle 2\rho_M,\eta\rangle.$

If this claim holds, then the complement of the inverse image of this Newton stratum under the projection $\pr_M:S^G_{\lambda,\eta}\rightarrow S^M_{\lambda,\eta} $  has dimension strictly smaller than $\langle 2\rho,\eta\rangle.$ By the above considerations, it contains the complement of $S^{G,[b(\lambda)]}_{\lambda,\eta}$. Thus it is enough to prove the above claim. Replacing $G$ by $M$ it is thus enough to show the theorem for the case that $G=M$, that $[b(\lambda)]$ is superbasic and that $\lambda^{\diamond}\geq \nu_{b(\lambda)}=\av_G(\lambda)$.

We use induction on the semisimple rank of $G$. If the semisimple rank is 0, then every point is basic, thus the assertion holds. Assume that the semisimple rank is positive. Let $P$ be a maximal standard parabolic subgroup of $G$ (and $M$ its standard Levi subgroup) such that $\av_M(\lambda)\geq \av_G(\lambda)=\nu_{b(\lambda)}$ is minimal among these vectors. For an example of $\lambda\geq \av_M(\lambda)\geq \nu_{b(\lambda)}$ for $G=\GL_7$ and $M=\GL_5\times\GL_2$ see Figure \ref{fig2}.

Since $b(\lambda)$ is superbasic, we have $\nu_{b(\lambda)}^{\sharp_M}\lneq \av_M(\lambda)$. Indeed, otherwise the basic class in $B(M)$ with image $\lambda^{\sharp_M}$ in $\pi_1(M)_{\Gamma}$ would be a reduction of $[b(\lambda)]$ to $M$.

{\setlength{\unitlength}{16 pt}
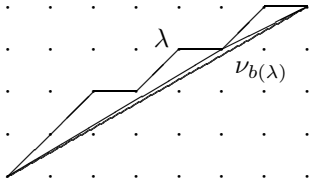
\begin{figure}[h]
\begin{center}
\begin{picture}(8,3.3)(0,0.5)
\multiput(0,0)(1,0){8}{$\cdot$}
\multiput(0,1)(1,0){8}{$\cdot$}
\multiput(0,2)(1,0){8}{$\cdot$}
\multiput(0,3)(1,0){8}{$\cdot$}
\multiput(0,4)(1,0){8}{$\cdot$}
\multiput(0,0.05)(0.035,0.02){201}{{\tiny{$\cdot$}}}

\put(0.1,0.17){\line(1,1){2}}
\put(2.1,2.17){\line(1,0){1}}
\put(3.1,2.17){\line(1,1){1}}
\put(4.1,3.17){\line(1,0){1}}
\put(5.1,3.17){\line(1,1){1}}
\put(6.1,4.17){\line(1,0){1}}

\put(0.1,0.17){\line(5,3){5}}
\put(5.1,3.17){\line(2,1){2}}
\put(5.4,2.6){$\nu_{b(\lambda)}$}
\put(3.5,3.2){$\lambda$}
\end{picture}
\end{center}
\caption{The minimal $\lambda$ and $\av_M(\lambda)$ for $b'$ of Newton slope $\tfrac{4}{7}$} \label{fig2}
\end{figure}
}

\noindent{\it Claim 2.} There is no $[b']\in B(G)$ with $\kappa_G(b')=\lambda^{\sharp_G}$ and $\nu_{b(\lambda)}\lneq \nu_{b'}\lneq \av_M(\lambda)$.

Assume that $[b']$ is a class violating the claim. Then $[b']$ is not basic. Let $M'$ be the centralizer of its Newton point. Then  $\nu_{b(\lambda)}^{\sharp_{M'}}\leq \kappa_{M'}(b')= \lambda^{\sharp_{M'}}$, and by minimality of $\lambda$, we have equality. 
In other words, $\nu_{b'}=\av_{M'}(\lambda)$. Let $M''$ be a maximal standard Levi subgroup containing $M'$ and such that $\av_{M''}(\lambda)\neq \av_G(\lambda)$. Replacing $b'$ by the basic class $[b'']$ of $B(M'')$ with $\kappa_{M''}(b'')=\lambda^{\sharp_{M''}}$ we have $\nu_{b''}\leq \nu_{b'}$ non-basic. Thus we may assume that $M'$ is maximal. But then the condition  $\nu_{b'}=\av_{M'}(\lambda)\lneq \av_M(\lambda)$ violates our choice of $P$ and $M$ above, which proves the claim.

The reduction $\E_{1}^P$ induces for every $x\in S_{\lambda,\eta}(C)$ a reduction $(\E_{1,x})_P$. By Lemma \ref{lemslopekappa}, its slope vector is $v=\av_M(-\lambda)$. We assume that $x$ is not in the basic Newton stratum.

We distinguish two cases and first consider those $x$ where $(\E_{1,x})_P\times^P M$ is a non-semi-stable $M$-bundle. Using the Iwasawa decomposition we write $x$ as $\lambda(\xi)mnG(B_{\dR}^+)/G(B_{\dR}^+)$ for $m\in M(B_{\dR})$ and $n\in N(B_{\dR})$ where $N$ is the unipotent radical of $P$. Then our condition is that $\lambda(\xi)mM(B_{\dR}^+)/M(B_{\dR}^+)$ is in the complement of the basic locus in $S_{\lambda,\eta}^M$. By induction, this is closed and of dimension less than $\langle 2\rho_M, \eta\rangle$. Thus the subspace of $x$ such that $(\E_{1,x})_P\times^P M$ is not semi-stable is the inverse image under $\pr_M$ of a subspace of dimension less than $\langle 2\rho_M, \eta\rangle$, and hence it is of dimension less that $\langle 2\rho, \eta\rangle$.

It remains to consider the locus where $(\E_{1,x})_P\times^P M$ is semi-stable. In other words, we consider the intersection of complement of the basic Newton stratum of $S_{\lambda,\eta}$ with the inverse image under $\pr_M$ of the $M$-basic Newton stratum in $S_{\lambda,\eta}^{M}$. By \cite[Cor.~2.9]{Chen} the Newton point $\nu_{b(\lambda)}$ of $\E_{1,x}$ is less than or equal to that of $((\E_{1,x})_P\times^P M)\times^M G$, which is $(-v)_{\dom}=(\av_M(\lambda))_{\dom}=\av_M(\lambda)$. By Claim 2, this implies that the Newton point of $\E_{1,x}$ is $(-v)_{\dom}$, hence $(\E_{1,x})_P$ is split. Furthermore, we assumed this Newton point to be non-basic. Thus, we can apply Corollary \ref{lemadomspl}, and obtain that $x\in P(\Qp)M(B_{\dR})G(B_{\dR}^+)/G(B_{\dR}^+)\cap S_{\lambda,\eta}(C)$. For sufficiently regular $\eta$, this is indeed a subspace of strictly smaller dimension.
\end{proof}

\begin{remark}
It would be interesting to know if the conclusion also holds for other $\lambda$, or to compute generic Newton points also for modifications of other $G$-bundles. However, the present result is all we need for Theorem \ref{thmclos} below.
\end{remark}

\subsection{Closures of Newton strata}

Let $G$ be again any connected reductive group over $F$. Recall from \cite{FGtorseurs} that we have a bijection between $B(G)$ and the points of $|\Bun_G|$.
\begin{thm}\label{thmclos}
Let $[b'],[b'']\in B(G)$. Then $\E_{b''}\in \overline{\{\E_{b'}\}}$ in $|\Bun_G|$ if and only if $[b']\leq [b'']$ with respect to the partial order on $B(G)$.
\end{thm}
\begin{proof}
By Section \ref{secadinn} we may assume that $G$ is quasi-split.

One direction is already known: if $\E_{b''}\in \overline{\{\E_{b'}\}}$, the semi-continuity properties of Kedlaya-Liu \cite[Thm.~7.4.5]{KL} and Scholze-Weinstein \cite[Cor.~22.5.1]{SW} yield $\nu_{b'}\leq \nu_{b''}$. Local constancy of the Kottwitz point as in \cite[Thm.~III.2.7]{FS} then proves that $[b']\leq [b'']$.

Assume conversely that $[b']<[b'']$. We want to show that $\E_{b''}\in \overline{\{\E_{b'}\}}$. Using induction we may assume that there is no $[\tilde b]\in B(G)$ with $[b']<[\tilde b]<[b'']$. By \cite[Thm.~7.4(ii)]{Chai}, this implies that there is a unique element of the reduced edge subset of $[b'']$ that is not contained in the reduced edge subset of $[b']$. Let $\alpha$ be the associated simple relative root of $G$, let $\alpha_0$ be a corresponding simple absolute root and $\beta_0=w_0(\alpha_0)$, a negative root.

Let $\lambda'\in X_*(T)$ be a minimal element satisfying $[b(\lambda')]=[b']$. Thus it has the properties discussed in Remark \ref{remhv21}. In particular, all such $\lambda'$ have the same value of $\langle\rho,\lambda'\rangle$. Let $\lambda''=\lambda'+\beta_0^{\vee}$. Then the reduced edge subset of $w_0(\lambda'')$ contains one new edge, corresponding to the simple root $\alpha$. By minimality of $\lambda'$ and our choice of $\alpha$, the reduced edge subset of $w_0(\lambda'')$ coincides with that of $[b'']$, which implies $[b(\lambda'')]=[b'']$. Furthermore, we have $\lambda''<\lambda'$ and $\langle\rho,\lambda''\rangle=\langle\rho,\lambda'\rangle-1$.

We claim that an element $\tilde\lambda\in X_*(T)$ with $[b(\tilde\lambda)]=[b'']$ is minimal with this property if and only if $\langle\rho,\tilde\lambda\rangle=\langle\rho,\lambda'\rangle-1$. To show this, we use the description of such elements in terms of the element $\lambda_{b''}\in X^*(\hat T)_{\Gamma}$ associated with $[b'']$. Recall from Remark \ref{remhv21} that $\lambda_{b''}$ is defined by the condition that the pairings with the fundamental weights are the minimal integers greater than or equal to the corresponding pairing for $\nu_{b''}$, and similarly for $\lambda_{b'}$. These integers coincide (for $b''$ resp.~$b'$) for all fundamental weights except for the one corresponding to $\alpha$, where they differ by 1. This proves the claim, which in turn implies that $\lambda''$ is minimal with $[b(\lambda'')]=[b'']$.

By the closure relations for semi-infinite cells, see Remark \ref{propmvne} (1), 
we have $S_{\lambda''}\subseteq \overline{S_{\lambda'}}$. By Theorem \ref{propkey}, the image of $S_{\lambda'}$ in $\Bun_G$ lies in the closure of $\{\E_{b'}\}$, and the image of $S_{\lambda''}$ contains $\{\E_{b''}\}$. The theorem follows.
\end{proof}
Theorem \ref{thmclos} determines the topology on $\Bun_G$ completely:
\begin{cor}\label{cortop}
Let $M\subset |\Bun_G|$. Then $\overline{M}=\bigcup_{\E\in M}\overline{\{\E\}}.$
\end{cor}
\begin{proof}
 By \cite[Thm.~III.2.7]{FS}, the Kottwitz map $|\Bun_G|\rightarrow \pi_1(G)_{\Gamma}$ is locally constant, thus it is enough to consider the inverse image $\Bun_{G,c}$ of any fixed element $c$ of the discrete set $\pi_1(G)_{\Gamma}$. We may replace $G$ by its adjoint group, and consider each simple factor of it separately. Thus we may assume that $G$ is adjoint and simple.

We want to show that for any subset $M\subset \Bun_{G}$, the closure $\overline M$ coincides with the union of the closures of the elements of $M$. For $[b]\in B(G)$, the set $$S_{[b]}:=\{[b']\in B(G)\mid \kappa_G(b')=\kappa_G(b);~ [b']\ngeq [b]\}$$ is the subset of $B(G)$ corresponding to $|\Bun_{G,\kappa_G(b)}|\setminus \overline{\{\E_{b}\}}$. Thus it is enough to prove that every set $S_{[b]}$ is finite.

Let $[b']\in S_{[b]}$. Since $\nu_{b'}-\nu_b$ is Galois-invariant, we can write it as a linear combination $\sum_i d_i\alpha_i^{\vee}$ of the simple relative coroots. Let $[b_0]$ be the unique basic class with $\kappa_G(b_0)=\kappa_G(b')$. We have $[b']\geq [b_0]$, hence the coefficients $d_i$ are bounded below by the corresponding coefficients of $\nu_{b_0}-\nu_b$, that is, independently of $[b']$. Since $[b']\ngeq [b]$, at least one $d_{i_0}$ for some $i_0$ is negative. It remains to show that $d_{i_0}<0$ implies an upper bound on all $d_i$. We proceed by induction on the distance of $\alpha_i$ and $\alpha_{i_0}$ in the Dynkin diagram. For the induction step we assume that $d_j<c$ for some $j$ and some $c$. Then
\begin{align*}
\langle \nu_{b'},\alpha_j\rangle&=\langle \nu_{b},\alpha_j\rangle+\langle \sum d_i\alpha_i^{\vee},\alpha_j\rangle\\
&=\langle \nu_{b},\alpha_j\rangle+ 2d_j+\langle \sum_{i'} d_{i'}\alpha_{i'}^{\vee},\alpha_j\rangle
\end{align*}
where the first sum is taken over all simple roots and the second only over the neighbors of $\alpha_j$ in the Dynkin diagram. The two summands $\langle \nu_{b},\alpha_j\rangle+ 2d_j$ are bounded above independently of $b'$, and $\langle \alpha_{i'}^{\vee},\alpha_j\rangle$ is negative for each $i'$. Since $\nu_{b'}$ is dominant, we have $\langle \nu_{b'},\alpha_j\rangle\geq 0$, which implies an upper bound on $\sum_{i'} d_{i'}(-\langle\alpha_{i'}^{\vee},\alpha_j\rangle )$, which we view as a linear combination of the $d_{i'}$ with positive coefficients. We saw above that each $d_{i'}$ is bounded below, thus this also implies an upper bound on each $d_{i'}$ individually. This completes the induction.
\end{proof}

We end this section by applying Theorem \ref{thmclos} to study closures of Newton strata in affine Schubert cells. Let $\{\mu\}$ be a conjugacy class of cocharacters of $G$ and let $b\in G(\breve{\Qp})$ be basic. Recall that for $[b']\in B(G)$ we denote by $\Gr_{G,\mu,b}^{ [b']}$ the Newton stratum for $[b']$ in the affine Schubert cell for $\mu$ (using modifications of $\E_b$).
\begin{cor}\label{corclosns}
Let $\mu$ and $b$ be as above and let $[b']\in \BGmub$. Then
$$\overline{\Gr_{G,\mu,b}^{ [b']}}=\coprod_{[b'']\geq [b']}\Gr_{G,\mu,b}^{ [b'']}.$$
\end{cor}
\begin{proof}
The proof of \cite[Prop.~2.11]{Hansen2} shows that $\overline{\Gr_{G,\mu,b}^{ [b']}}$ is a union of those Newton strata corresponding to $[b'']$ with $\E_{b''}\in \overline{\{\E_{b'}\}}$. (In \cite[Prop.~2.11]{Hansen2} this statement is shown for $\mu$ minuscule and $[b]=[1]$, but the present more general statement is shown by the same argument.) Then the corollary follows from Theorem \ref{thmclos}.
\end{proof}

\section{Newton strata in the weakly admissible locus}\label{secnswa}

In this section, we always fix a geometric conjugacy class $\{\mu\}$ of cocharacters $\mathbb G_m\rightarrow G_{\overline {\Qp}}$. By $\mu$ we denote the representative in $X_*(T)_{\dom}$.

Furthermore, we fix a basic element $b\in G(\breve{\Qp})$.

\subsection{The Hodge-Newton decomposition for modifications of $G$-bundles}\label{secHNdec}

The following is a variant of the definition of Hodge-Newton-decomposability from \cite[Def.~3.1]{Chen}. It requires a little bit more than the notion of HN-reducibility in \cite[Def.~4.28]{RV}.
\begin{definition}\label{defHNdec}
\begin{enumerate}
\item Let $[b']\in B(G)$ and $\delta\in X_*(A)_{\Q,\dom}$ with $\nu_{b'}\leq \delta$. Then $(G,[b'],\delta)$ is Hodge-Newton decomposable if there is a proper standard Levi subgroup $M$ of the quasi-split inner form $H$ of $G$ containing the centralizer of $\nu_{b'}$ and such that $\delta^{\diamond}-\nu_{b'}\in \langle \Phi_{0,M}^{\vee}\rangle_{\Q}$. 
Otherwise, the triple is called Hodge-Newton indecomposable.
\item Let $\mu\in X_*(T)_{\dom}$ and $[b]\in B(G)$ basic. Then $[b']\in \BGmub$ is Hodge-Newton decomposable if the triple $(G,[b'],\nu_b(\mu^{-1,\diamond})_{\dom})$ is Hodge-Newton decomposable. Otherwise, $[b']$ is called Hodge-Newton indecomposable.
\end{enumerate}
\end{definition}
\begin{ex}
The basic element of $\BGmub$ is Hodge-Newton indecomposable. In particular, every set $\BGmub$ contains a Hodge-Newton indecomposable element.
\end{ex}
\begin{remark}\label{rem73}
Let $[b']\in B(G)$ and $\delta\in X_*(A)_{\Q,\dom}$ with $\nu_{b'}\leq \delta$. Then $(G,[b'],\delta)$ is Hodge-Newton decomposable if and only if $(G_{\ad},[b'_{\ad}],\delta_{\ad})$ is Hodge-Newton decomposable, and analogously for $\mu\in X_*(T)_{\dom}$, a basic element $[b]\in B(G)$ and $[b']\in \BGmub$.

Let $b_0\in G(\breve{\Qp})$ be basic. Then $(G,[b'],\delta)$ is Hodge-Newton indecomposable if and only if $(G_{b_0},[b'b_0^{-1}],\delta\nu_{b_0}^{-1})$ is Hodge-Newton indecomposable. Similarly, $[b']\in \BGmub$ is Hodge-Newton indecomposable if and only if $[b'b_0^{-1}]\in B(G_{b_0},\mu,bb_0^{-1})$ is Hodge-Newton indecomposable.
\end{remark}

\begin{lemma}\label{corhndecstrsm}
Assume that $G_{\ad}$ is simple and that $[b']\in B(G)$ is not basic. Let $\delta\in X_*(A)_{\Q,\dom}$ with $\nu_{b'}\leq \delta$. Then $(G,[b'],\delta)$ is Hodge-Newton decomposable if and only if there is a proper standard Levi subgroup $M$ of the quasi-split inner form $H$ of $G$ such that $\delta-\nu_{b'}\in \langle \Phi_{0,M}^{\vee}\rangle_{\Q}$.
\end{lemma}
\begin{proof}
If $(G,[b'],\delta)$ is Hodge-Newton decomposable, the condition is automatically satisfied. Assume conversely that $(G,[b'],\delta)$ satisfies the condition for some $M$. We may choose $M$ to be minimal with this condition. Then $\Delta'=\Delta_{0,M}$ is a proper subset of $\Delta_0$ with $\delta-\nu_{b'}\in \langle \Delta'\rangle_{\Q}$. If $\Delta'$ is empty, then $\delta=\nu_{b'}$, in particular $(G,[b'],\delta)$ satisfies the definition of Hodge-Newton decomposability (for the centralizer of $\nu_{b'}$). Assume that this is not the case. Let $\alpha\in \Delta_0\setminus\Delta'$ be such that one of its neighbors in the Dynkin diagram is contained in $\Delta'$.  We write $\delta-\nu_{b'}$ as a non-negative linear combination of positive coroots with coefficients $c_{\beta}$ for $\beta\in \Phi^+\cap \langle \Delta'\rangle_{\Q}$. We have $\langle\alpha, \nu_{b'}\rangle=\langle \alpha,\delta\rangle-\langle\alpha, \delta-\nu_{b'}\rangle$. Since $\delta$ is dominant, the first summand on the right hand side is non-negative. The pairing $\langle \alpha,\delta-\nu_{b'}\rangle$ is negative since it is a linear combination of the $\langle \alpha,\beta^{\vee}\rangle$ for all neighbors $\beta$ of $\alpha$ in the Dynkin diagram with coefficients $c_{\beta}\geq 0$, and at least one of the $c_{\beta}$ is non-zero. Thus $\langle \alpha,\nu_{b'}\rangle> 0$, and the maximal standard Levi subgroup of $G$ corresponding to $\alpha$ contains the centralizer of $\nu_{b'}$, and also $M$. In particular, $(G,[b'],\delta)$ satisfies the condition for Hodge-Newton decomposability for this subgroup.
\end{proof}

We assume for the next lemma that $G$ is quasi-split, and fix a maximal split torus $A$ and a Borel subgroup containing its centralizer. Let $[b_1],[b_2]\in B(G)$ with $\kappa_G(b_1)=\kappa_G(b_2)$. By \cite[Thm.~6.5]{Chai}, the join of $[b_1],[b_2]$ in $B(G)$ with respect to $\leq$ exists, i.e.~there is a unique minimal element in $\mathcal N(G)$ which is greater than or equal to $[b_1]$ and $[b_2]$.
\begin{lemma}\label{lemjoinprecs}
Let $G$ be quasi-split. Let $[b_1],[b_2]\in B(G)$ and $\delta\in X_*(A)_{\Q,\dom}$ with $\kappa_G(b_1)=\kappa_G(b_2)$ and $[b_1],[b_2]\in  B(G)$ with $\nu_{b_1},\nu_{b_2}\leq \delta$ two Hodge-Newton indecomposable classes. Let $[b']$ be the join of $[b_1],[b_2]$. Then $[b']$ is also Hodge-Newton indecomposable.
\end{lemma}

\begin{proof}
Since $\nu_{b_1}\leq \delta$ and $\nu_{b_2}\leq \delta$, we have $\nu_{b'}\leq\delta$.
Assume that $[b']$ is Hodge-Newton decomposable. Thus there is an $\alpha\in \Delta_0$ with $\langle \alpha,\nu_{b'}\rangle> 0$ such that
\begin{equation}\label{eqjoinprecs}
\langle\tilde\omega_{\alpha},\delta-\nu_{b'}\rangle=0.
\end{equation}

In the language of \cite[6]{Chai}, $\langle\alpha_i,\nu_{b'}\rangle> 0$ for some $\alpha_i\in \Delta_0$ means that $\alpha_i$ corresponds to an element of the reduced edge subset of $[b']$, which is by \cite[Thm.~6.5]{Chai} a subset of the union of the reduced edge subsets of $[b_1]$ and $[b_2]$. In particular, we have for these $\alpha_i$ that $\langle \omega_i,\nu_{b'}\rangle=\langle \omega_i,\nu_{b_j}\rangle$ and $\langle \alpha_i,\nu_{b_j}\rangle>0$ for some $j\in \{1,2\}$. Together with \eqref{eqjoinprecs} this is a contradiction to the Hodge-Newton indecomposability of $[b_j]$.
\end{proof}

\begin{cor}\label{cormaxhnind}
The set $\BGmub$ contains a unique maximal Hodge-Newton indecomposable element.
\end{cor}
\begin{proof}
By Remark \ref{rem73} we may replace $G$ by the quasi-split inner form of its adjoint group and thus assume that $G$ is quasi-split. Let $[b_{\max}]$ be the join of the finite, non-empty set of Hodge-Newton indecomposable elements of $\BGmub$. By Lemma \ref{lemjoinprecs}, $[b_{\max}]$ is Hodge-Newton indecomposable.
\end{proof}

\begin{lemma}\label{lemcondhonedec}
Assume that $G$ is quasi-split. Let $\mu\in X_*(T)_{\dom}$ and let $[b]\in B(G)$ be basic. Let $[b']\in \BGmub$. Assume that $M$ is a proper standard Levi subgroup of $G$ such that there is a reduction $[b'_M]_M$ of $[b']$ to $M$ with $M$-dominant Newton point $\nu_{b'}$ and $\nu_b+(-\mu^{\diamond})_{\dom}-\nu_{b'}\in \langle \Phi_{0,M}^{\vee}\rangle_{\Q}$. Then there is a reduction $[b_M]_M$ of $[b]$ to $M$ such that $[b'_M]_M\in B(M,w_{0,M}w_0\mu,b_M).$

This applies in particular to all Hodge-Newton decomposable $[b']\in \BGmub$ and all proper Levi subgroups $M$ of $G$ containing the centralizer of $\nu_{b'}$ such that $\nu_b+(-\mu^{\diamond})_{\dom}-\nu_{b'}\in \langle \Phi_{0,M}^{\vee}\rangle_{\Q}$.
\end{lemma}
\begin{proof}
Let $[\tilde b]_M\in B(M)$ be basic with $\kappa_M(\tilde b)=\kappa_M(b'_M)-(\mu^{-1})_{\dom}^{\sharp_M}$. Since $\kappa_G(b'_M)=\kappa_G(b')=\kappa_G(b)-\mu^{\sharp}$, we have $\kappa_G(\tilde b)=\kappa_G(b)\in\pi_1(G)_{\Gamma}$. From $\nu_{b'_M}^{\sharp_M}=(\nu_b+(\mu^{-1})_{\dom})^{\sharp_M}$ we obtain $\nu_b^{\sharp_M}=\nu_{b'_M}^{\sharp_M}-(\mu^{-1})_{\dom}^{\sharp_M}=\nu_{\tilde b}^{\sharp_M}$ in $\pi_1(M)_{\Gamma,\Q}$. Since the kernel of $\pi_1(M)_{\Gamma}\rightarrow \pi_1(G)_{\Gamma}$ is torsion free, this implies $\kappa_M(b)=\kappa_M(\tilde b)\in \pi_1(M)_{\Gamma}$. Hence the two basic classes agree and $[\tilde b]_M$ is a reduction of $[b]$ to $M$.

It remains to prove that  $[b'_M]_M\in B(M,w_{0,M}w_0\mu,b_M).$ By the previous step we have $\kappa_M(b'_M)=\kappa_M(b_M)+(-\mu)_{\dom}^{\sharp_M}$. By assumption $\nu_{b'_M}=\nu_{b'}\leq_G \nu_b(\mu^{-1,\diamond})_{\dom}$ and both sides have the same image in $\pi_1(M)_{\Gamma,\Q}$. Hence, their difference is a non-negative linear combination of positive coroots of $M$, and $\nu_{b'_M}\leq_M \nu_b(\mu^{-1,\diamond})_{\dom}$.

The second assertion follows since every $[b']\in B(G)$ has a reduction to the centralizer of its Newton point.
\end{proof}

\begin{prop}[Hodge-Newton decomposition for modifications of $G$-bundles]\label{prophonedec}
Assume that $G$ is quasi-split. Let $\mu\in X_*(T)_{\dom}$ and let $b\in G(\breve{\Qp})$  be basic. Let $M$ be a standard Levi subgroup of $G$ such that there are reductions $[b'_M]_M$ and $[b_M]_M$ of $[b']$ and $[b]$ to $M$ whose $M$-dominant Newton points coincide with the $G$-dominant Newton points $\nu_{b'}$ resp.~$\nu_{b}$. Assume that $[b'_M]_M\in B(M,w_{0,M}w_0\mu,b_M)$. 
Let $M^*={}^{w_0}M$, $b'_{M^*}={}^{w_0}b'_M$, and $b_{M^*}={}^{w_0}b_M$ be the conjugates under $w_0$. Then the map $$G_b(F)\times\Gr_{M^*,\mu,b_{M^*}}^{[b'_{M^*}]_{M^*}}(C)\rightarrow\Gr_{G,\mu,b}^{[b']}(C)$$ induced by the inclusion $M^*\hookrightarrow G$ and the $G_b(F)$-action on $\Gr_{G,\mu,b}^{[b']}$ is surjective.
\end{prop}
Notice that the assumptions of this proposition are slightly weaker than requiring that $[b']$ has to be Hodge-Newton decomposable (or Hodge-Newton decomposable for the Levi subgroup $M$) since we replace the assumption that $M$ contains the centralizer of $\nu_{b'}$ by the assumption on existence of a reduction of $[b']$ to $M$.

For the proof we need the following general group-theoretic lemma, which is also part of Kottwitz's proof of the classical Hodge-Newton decomposition for unramified groups in \cite{KHN}. We will apply it to $L=B_{\dR}(F)$.
\begin{lemma}\label{lemkhn}
Let $G$ be an unramified reductive group over a complete discretely valued field $L$ with valuation ring $\O_L$ and uniformizer $t$. Let $T$ be an unramified maximal torus and let $B$ be a Borel subgroup of $G$ containing the centralizer of $T$. Let $\mu\in X_*(T)_{\dom}$. Let $U$ be the unipotent radical of $B$. Let $K$ be a hyperspecial maximal subgroup of $G$. Let $P$ be a standard parabolic subgroup of $G$ with Levi factor $M$. Let $U_M$ and $K_M$ be the induced subgroups of $M$.

Let $\lambda\in X_*(T)$ with $\lambda^{\sharp_M}=w_0(\mu)^{\sharp_M}\in \pi_1(M)_{\Gamma}$. Then for every unramified extension $L'$ of $L$, we have $$ U(L')\lambda(t)K\cap K\mu(t)K\subseteq K_Mw_0(\mu)(t)K.$$
\end{lemma}
\begin{proof} The left hand side is empty unless $\lambda_{\dom}\leq \mu$, so we assume this. By \cite[Lemma 4.2]{KHN}, this condition together with $\lambda^{\sharp_M}=w_0(\mu)^{\sharp_M}$  implies that $\lambda$ and $\mu$ have the same image in $\pi_1(M)$ (without taking Galois-coinvariants). Then by \cite[Lemma 2.2]{KHN} and its proof (which is still valid for unramified groups although the lemma is only stated for split groups), the above assertion follows.
\end{proof}

\begin{proof}[Proof of Proposition \ref{prophonedec}]
Let $P^*$ be the standard parabolic subgroup of $G$ with Levi factor $M^*$. Let $x\in \Gr_{G,\mu,b}^{[b']}(C).$ Let $(\E_{b,x})_{P^*}=\E_{b'_{M^*}}\times^{M^*} {P^*}$ be the reduction to ${P^*}$ induced by the reduction $b'_{M^*}$ of $[b']$ to ${M^*}$. Then by \eqref{rem24}, this induces a reduction $(\E_b)_{P^*}$ of $\E_b$ to $P^*$. Let $v$ be the associated slope vector. Then since $\E_b$ is semi-stable $v\leq -\nu_{b}$.

Let $\lambda\in X_*(T)$ with $x\in S_{\lambda}(C)$. Non-emptiness of $S_{\lambda}\cap \Gr_{G,\mu}$ implies by Remark \ref{propmvne} (3) that $w_0(\lambda)\leq w_0(-\mu)$, hence $-\lambda\leq\mu$. By Lemma \ref{lemkey}, $(\E_{b,x})_{P^*}\times^{P^*}M^*\cong ((\E_b)_{P^*}\times^{P^*}{M^*})_{y}$ for $y=\pr_{M^*}(x)\in S_{\lambda}^{M^*}(C)$.

Since $\nu_{b'_{M}}^{\sharp_{M}}=\nu_b^{\sharp_{M}}(\mu^{-1})_{\dom}^{\sharp_{M}}$, we obtain by $w_0$-conjugation and multiplication by $-1$ that $$-\nu_b^{\sharp_{M^*}}+\mu^{\sharp_{M^*}}=(w_0(-\nu_{b'_M}))^{\sharp_{M^*}}=c_1(\E^{M^*}_{b'_{M^*}})=c_1((\E_{b})_{P^*}\times^{P^*}M^*)-\lambda^{\sharp_{M^*}}=(v-\lambda)^{\sharp_{M^*}}.$$ Since $v\leq -\nu_b$ and $\lambda\leq -\mu$, this implies $v^{\sharp_{M^*}}=-\nu_b^{\sharp_{M^*}}$ and $\lambda^{\sharp_{M^*}}=-\mu^{\sharp_{M^*}}$.

The first of these equalities means that $(\E_b)_{P^*}$ is split, hence $(\E_b)_{P^*}\times^{P^*}M^*$ is isomorphic to $\E_{w_0(b_M)}$. By Proposition \ref{lemlamlem}(3) there is an automorphism of $\E_b$ (i.e., an element of $G_b(F)$) identifying $(\E_b)_{P^*}$ with the reduction $\E_{w_0(b_M)}^{P^*}$ of $\E_b$. Thus changing $x$ within its $G_b(F)$-orbit, we may assume that the two reductions are equal. Since $\lambda^{\sharp_{M^*}}=-\mu^{\sharp_{M^*}}$ (and $-\mu$ is anti-dominant), we can apply Lemma \ref{lemkhn} (using an isomorphism $B_{\dR}^+(C)\cong C[\![ \xi]\!]$), which then implies that $x$ is in the image of $\Gr_{M^*,\mu,b_{M^*}}^{[b'_{M^*}]_{M^*}}(C)$.
\end{proof}

\begin{proof}[Proof of Theorem \ref{thmmain2}, (1)$\Rightarrow$(2)]
Let $[b']\in \BGmub$ be Hodge-Newton decomposable, and let $x\in \Gr_{G,\mu,b}^{[b']}(C)$. We want to show that $x$ is not weakly admissible. As usual, we may assume that $G$ is adjoint and quasi-split. Let $M\subset G$ be as in the definition of Hodge-Newton decomposability. Let $b_M$ be a reduction of $b$ to $M$ as in Lemma \ref{lemcondhonedec}. We use the notation of Proposition \ref{prophonedec}. Let $(j,x')\in G_b(F)\times \Gr_{M^*,\mu,w_0(b_M)}^{[b'_{M^*}]_{M^*}}(C)$ map to $x$ under the surjection of Proposition \ref{prophonedec}. Let $(\E_{b,x})_{P^*}$ be the reduction of $\E_{b,x}$ corresponding to the reduction $\E_{b_{M^*}}^{M^*}\times^{M^*} P^*$ of $\E_b$, composed with the automorphism $j$ of $\E_b$. By Lemma \ref{lemkey} we have $(\E_{b,x})_{P^*}\times^{P^*} M^*\cong \E_{b_{M^*},x'}^{M^*}$. In particular, the slope vector of this reduction is $\av_{M^*}(-\nu_{b}+\mu)\nleq \av_G(-\nu_b+\mu)$. Thus the modification is not weakly admissible.
\end{proof}

\subsection{Newton strata in the weakly admissible locus}\label{secconc}
From the comparisons of Newton strata in Section \ref{secadinn} together with Lemma \ref{corwatoqsp} we obtain
\begin{cor}\label{lemma}
Let $\{\mu\}$ be a conjugacy class of cocharacters of $G$ and let $b\in G(\breve{\Qp})$. Let $[b']\in B(G,\mu,b)$.
\begin{enumerate}
\item Then $\Gr_{G,\mu,b}^{\wa}\cap \Gr_{G,\mu,b}^{[b']}\neq \emptyset$ if and only if for the corresponding images in $G_{\ad}$ we have $\Gr_{G_{\ad},\mu_{\ad},b_{\ad}}^{\wa}\cap \Gr_{G_{\ad},\mu_{\ad},b_{\ad}}^{[b'_{\ad}]}\neq \emptyset$.
\item Let $b_0\in G(\breve{\Qp})$ be basic. Then  $\Gr_{G,\mu,b}^{\wa}\cap \Gr_{G,\mu,b}^{[b']}\neq \emptyset$ if and only if $\Gr_{G_{b_0},\mu,bb_0^{-1}}^{\wa}\cap \Gr_{G_{b_0},\mu,bb_0^{-1}}^{[b'b_0^{-1}]}\neq \emptyset$.
\end{enumerate}
\end{cor}

 The main step in the proof of the second implication in Theorem \ref{thmmain2} is concerned with the $\sigma$-conjugacy class considered in Corollary \ref{cormaxhnind}.

\begin{thm}\label{thmmaxi}
Let $[b']\in \BGmub$ be the unique maximal Hodge-Newton indecomposable element. Then $\Gr_{G,\mu,b}^{\wa}\cap \Gr_{G,\mu,b}^{[b']}\neq \emptyset$.
\end{thm}

We first use this theorem to finish the proof of Theorem \ref{thmmain2}.
\begin{proof}[Proof of Theorem \ref{thmmain2}, (2)$\Rightarrow$(1)]
Assume that $[b_0]\in \BGmub$ is Hodge-Newton indecomposable and that $\Gr_{G,\mu,b}^{[b_0]}$ does not intersect the weakly admissible locus. Then $\Gr_{G,\mu,b}^{[b_0]}$ is contained in its closed complement, and thus the same holds for $\overline{\Gr_{G,\mu,b}^{[b_0]}}$ where we take the closure within the affine Schubert cell for $\mu$. However, $[b_0]\leq [b_{\max}]$ where $[b_{\max}]$ is the maximal Hodge-Newton indecomposable element. Hence by Corollary \ref{corclosns}, $\Gr_{G,\mu,b}^{[b_{\max}]}\subset  \overline{\Gr_{G,\mu,b}^{[b_0]}}$ is contained in the complement of $\Gr_{G,\mu,b}^{\wa}$, contradicting Theorem \ref{thmmaxi}.
\end{proof}

\begin{proof}[Proof of Theorem \ref{thmmaxi}]
By Corollary \ref{lemma} we may replace $G$ by the quasi-split inner form of its adjoint group, and also consider each simple factor separately. Thus from now on we assume that $G$ is adjoint, simple, and quasi-split. We subdivide the proof into several steps.\\

\noindent{\it Step 1.} Let $C$ be a complete algebraically closed field extension of ${\Qp}$. As first step we show that for every $x\in \Gr_{G,\mu,b}^{[b']}(C)$ that is not weakly admissible, there is a proper standard parabolic subgroup $P$ (depending on the point $x$), and a reduction $b_{M}$ of $b$ to its standard Levi subgroup $M$ such that for the associated reduction $(\E_{b,x})_{P}$ we have
\begin{enumerate}
\item $(\E_{b,x})_P$ is split,
\item the slope vector of $(\E_{b,x})_P$ is non-basic and dominant.
\end{enumerate}

\begin{remark}\label{remint}
In the above context let $[b'_M]\in B(M)$ be such that $(\E_{b,x})_P\times^P M\simeq \E^M_{b'_M}$. Then (1) is equivalent to $[b'_M]_G=[b']$. We do not require the ($M$-dominant) Newton point of $[b'_M]_M$ to be $G$-dominant, but just the slope vector as in (2), i.e.~$\av_M(-\nu_{b'_M})$, is $G$-dominant.
\end{remark}

Assume that $x\in \Gr_{G,\mu,b}^{[b']}(C)$ is not weakly admissible. Then there is a standard parabolic subgroup $P$ (depending on the point $x$) and a reduction $b_{M}$ of $b$ to its standard Levi subgroup $M$ such that the slope vector associated with the reduction of $\E_{b,x}$ corresponding to the reduction $\E_{b_M}^P$ of $\E_b$ satisfies $v_x\nleq \av_G(v_x)$. We let $P$ and $(\E_{b,x})_{P}$ be such that $v_x\in\mathcal N(G)$ is maximal (for the partial order $\leq$) among the possible elements for the given point $x$. Then $v_x$ is non-basic and the same argument as for the claim in the proof of Theorem \ref{propclpt} shows that $v_x$ is dominant.

Again, we write $M^*=w_0(M)$, and similarly for other Levi subgroups, and let $\nu_{x,P}=v_x^*=w_0(-v_x),$ an element which is central in $M^*$. Then also $\nu_{x,P}$ is $G$-dominant.

Replacing $P$ by a larger parabolic subgroup we may assume that $\langle \alpha, \nu_{x,P}\rangle>0$ for all positive $\alpha$ that are not in $M^*$. 

Let $HN(\E_{b,x,M})$ be the Harder-Narasimhan polygon of $\E_{b,x,M}=(\E_{b,x})_{P}\times^{P}M$. Let $\nu_{M}=w_0(-HN(\E_{b,x,M}))$ and $\nu_{M,G}=(\nu_{M})_{G-\dom}$. Then $\nu_{M,G}$ coincides with the Newton polygon of $(\E_{b,x,M})\times^{M}G$.

\noindent{\it Claim 1.}  $\nu_{b'}\leq\nu_{M,G}\leq \nu_b(\mu^{-1})_{\dom}^{\diamond}$.

The first inequality follows from \cite[Cor.~2.9]{Chen}. Using the Iwasawa decomposition we have a representative of $x$ of the form $mn$ for $m\in M(B_{\dR})$ and $n\in N(B_{\dR})$. Since $m$ is the retraction of $x$ to $M$, we have $m\in \Gr_{M,\mu_1}(C)$ for some $M$-dominant $\mu_1$ with $\mu_{1,\dom}\leq \mu$. By Lemma \ref{lemkey}, $\E_{b,x,M}=(\E_{b,x})_{P}\times^{P}M= \E_{b_{M},m}$.

Thus $\nu_{M}\leq_{M^*} \nu_b(w_0(\mu_1^{-1}))^{\diamond}$ where $\leq_{M^*}$ denotes the partial order associated with the group $M^*$. By \cite[Thm.~5.2 (1)]{minnp}, the $G$-dominant representatives of these two elements are then related by $\leq$. This proves Claim 1.

Hence $\E_{b,x,M}\times^{M}G\cong\E_{b''}$ for some $[b'']\in \BGmub$ with $\nu_{b''}=\nu_{M,G}$ and $[b']\leq[b'']$. By the maximality of $[b']$, either $[b']=[b'']$ or $[b'']$ is Hodge-Newton decomposable.

\noindent{\it Claim 2.} If $[b'']$ is Hodge-Newton decomposable, then also $(M^*,\nu_{M},\nu_b(w_0(\mu_1^{-1}))^{\diamond})$ is Hodge-Newton decomposable.

We write $\zeta_1=\nu_b(w_0(\mu_1^{-1}))^{\diamond}$ and $\zeta=\nu_b(\mu^{-1,\diamond})_{\dom}$. Then $\zeta_1=w(\zeta')$ for some $w\in W$ and $\zeta'\leq\zeta$ dominant. Since $\zeta_1$ is $M^*$-dominant and $\zeta'$ is $G$-dominant, we may choose $w\in {}^{M^*}W$, i.e.~a shortest representative of $W_{M^*}w$. Similarly, we have $\nu_M=w'(\nu_{M,G})$ with $w'\in {}^{M^*}W$. Since $[b'']$ is Hodge-Newton decomposable, there is a standard Levi subgroup $L$ of $G$ containing the centralizer of $\nu_{M,G}$ and such that $(\nu_{M,G})^{\sharp_L}=\zeta^{\sharp_L}$. We have $\nu_M\leq_{M^*}\zeta_1.$ Thus $w'\in {}^{M^*}W$ implies that $\nu_{M,G}\leq (w')^{-1}(\zeta_1)=(w')^{-1}w(\zeta')$. In particular, $(\zeta')^{\sharp_L}\leq\zeta^{\sharp_L}=(\nu_{M,G})^{\sharp_L}\leq ((w')^{-1}w(\zeta'))^{\sharp_L}$ where the partial order on $\pi_1(L)_{\Gamma,\Q}$ is induced by the one on $X_*(A)_{\Q}$. Thus by Lemma \ref{lems455} (for $M_1$ the centralizer of $\zeta'$ and $M_2=L$) we have $(w')^{-1}w\in W_LW_{\zeta'}$, and  \begin{equation}\label{eqcl2}
(\zeta')^{\sharp_L}=\zeta^{\sharp_L}=(\nu_{M,G})^{\sharp_L}= ((w')^{-1}w(\zeta'))^{\sharp_L}.
\end{equation}

Conjugating by $w'$ we obtain $\nu_M^{\sharp_{{}^{w'}L}}=\zeta_1^{\sharp_{{}^{w'}L}}$ where ${}^{w'}L$ is the image of $L$ under conjugation by $w'$. We claim that ${}^{w'}L\cap {M^*}$ is also a standard Levi subgroup. Indeed, one first observes that ${}^{w'}L\cap {M^*}$ contains the fixed maximal torus. Let $\alpha$ be a positive root in ${}^{w'}L\cap {M^*}$ and assume that $\alpha=\alpha_1+\alpha_2$ for positive roots $\alpha_i$ in $G$. Since $M^*$ is standard, $\alpha_1$ and $\alpha_2$ are positive roots of $M^*$. Since $w'\in {}^{M^*}W$, the inverse images of $\alpha,\alpha_1$ and $\alpha_2$ under $w'$ are positive, and by assumption $(w')^{-1}(\alpha)$ is a root in $L$. But $L$ is standard, so $(w')^{-1}(\alpha_1), (w')^{-1}(\alpha_2)$ are also roots of $L$, which implies that $\alpha_1,\alpha_2$ are roots of ${}^{w'}L\cap {M^*}$, and finishes the proof of the claim. 

The subgroup $L$ contains the stabilizer of $\nu_{M,G}$, hence ${}^{w'}L\cap {M^*}$ contains the stabilizer of $\nu_M$ in ${M^*}$. To show that the Levi subgroup ${}^{w'}L\cap {M^*}$ is as in the definition of Hodge-Newton decomposability for $({M^*},\nu_b(w_0(\mu_1^{-1}))^{\diamond},\nu_{M})$, it remains to show that it is a proper subgroup of ${M^*}$. Assume that this is not the case. Then ${M^*}\subseteq {}^{w'}L$. Hence
$$\nu_{x,P}^{\sharp_{{}^{w'}L}}=\nu_M^{\sharp_{{}^{w'}L}}=\zeta_1^{\sharp_{{}^{w'}L}}=(w(\zeta'))^{\sharp_{{}^{w'}L}}= ((w')^{-1}w(\zeta'))^{\sharp_L}\overset{\eqref{eqcl2}}{=}(\zeta')^{\sharp_L},$$ where we identify $\pi_1(L)_{\Gamma,\Q}$ and $\pi_1({}^{w'}L)_{\Gamma,\Q}$ via conjugation by $w'$. Since $\nu_{x,P}$ is dominant, we have $\nu_{x,P}^{\sharp_L}\geq ((w')^{-1}(\nu_{x,P}))^{\sharp_L}=(\nu_{x,P})^{\sharp_{{}^{w'}L}}$. On the other hand, Lemma \ref{corhndecstrsm} together with Hodge-Newton indecomposability would imply $\nu_{x,P}^{\sharp_L}\leq \nu_{M}^{\sharp_L}\lneq (\zeta')^{\sharp_L}$, contradiction. This finishes the proof of Claim 2.

Assume that $[b'']$ is Hodge-Newton decomposable. From the Hodge-Newton decomposition (Prop.~\ref{prophonedec}) for $M$ we obtain a standard parabolic subgroup $\tilde P$ of $M$ together with a reduction of $[b_M]$ to its standard Levi subgroup $\tilde M$, and a reduction $\E_{b,x,\tilde P}$ of $\E_{b,x,M}$ to $\tilde P$ induced by the reduction of $\E_b$ to $\tilde P$ via \eqref{rem24}. By \cite[Proof of Lemma 6.4]{CFS}, $\E_{b,x,\tilde P}$ corresponds to a reduction of $\E_{b,x}$ to $\tilde P B$, which has the same slope vector $\tilde v$ as $\E_{b,x,\tilde P}$. It satisfies $v_x\lneq\tilde v$. This contradicts the maximality of $v_x$.

Thus $[b'']$ cannot be Hodge-Newton decomposable, and we have $[b']=[b'']$. Then $\E_{b,x,M}$ is a reduction of $\E_{b,x}$ to $M$. This finishes the proof of Step 1.\\

Consider now a proper parabolic subgroup $P$ of $G$ together with reductions $b_M$ and $b'_M$ of $b$ and $b'$ to $M$ such that $\av_{M}(\nu_{b'_M})$ is $G$-dominant and non-basic. Consider the subset $Z$ of $\Gr_{G,\mu,b}^{[b']}(C)$ of $x$ such that for the reduction $(\E_{b,x})_P$ of $\E_{b,x}$ associated with the reduction $\E_{b_M}^P$ of $\E_b$ we have $(\E_{b,x})_P\times^P M\simeq \E_{M,b'_M}$. Let $g$ be the element describing the reduction $b_M$ of $b$. Then $x\in Z(C)$ if and only if the following condition holds. Let $g^{-1}x\in S_{\lambda}(C)$ for some $\lambda$. Then $g^{-1}x\in \Gr_{G,\mu}$ implies $\lambda_{\dom}\leq(-\mu)_{\dom}$. Furthermore, from the computation of the slope vector in Lemma \ref{lemslopekappa} we obtain that $\av_M(-\lambda)$ non-basic and dominant. Finally, $\pr_M(g^{-1}x)\in \Gr_{M,(-\lambda)_{M-\dom},b_M}^{[b'_M]}$. Thus $Z$ can be described as the intersection of $\Gr_{G,\mu,b}^{[b']}$ with the union of the translates of all finitely many semi-infinite cells $S_{\lambda}$ as above and the preimage of $\Gr_{M,(-\lambda)_{M-\dom},b_M}^{[b']}$ under $\pr_M$. In particular, it is a locally spatial diamond.

By Step 1, the complement of the weakly admissible locus in $\Gr_{G,\mu,b}^{[b']}$ is a profinite union over all $Z$ for all possible choices of $P$ (finitely many), reductions $b_M$ (parametrized by $G_b(\Qp)/P_b(\Qp)$), and reductions $b'_M$ as above.

We consider the diagram of period maps
\begin{displaymath}
    \xymatrix{
       &\mathcal{M} ( G, \mu^{-1}, b, b')_{\infty}(C)\ar@{>>}[dl]_{\pi_{\dR}}\ar@{>>}[dr]^{\pi_{{\rm HT}}}& \\
\Gr_{G, \mu^{-1},b'}^{[b]} (C)&&\Gr_{G,\mu,b}^{[b']}(C)}
\end{displaymath}
where $\mathcal{M} ( G, \mu^{-1}, b, b')_{\infty}$ is the moduli space of modifications of type $\mu^{-1}$ between $\E_b$ and $\E_{b'}$. The basic Newton stratum $\Gr_{G, \mu^{-1},b'}^{[b]}$ is open in $\Gr_{G, \mu^{-1}}$ and hence of dimension $\langle 2\rho, \mu\rangle$. The map $\pi_{\dR}$ is pro-\'etale.

For each of the pro-finitely many subspaces $Z$ we consider $\pi_{\dR}(\pi_{\HT}^{-1}(Z))$. It is enough to show the following claim.

\noindent{\it Claim 3.} Any such $\pi_{\dR}(\pi_{\HT}^{-1}(Z))$ defines a subset which is a locally spatial diamond of dimension strictly less than $\dim \Gr_{G, \mu^{-1},b'}^{[b]} =\langle 2\rho, \mu\rangle$.

Fix some $Z$. Replacing $b$ by a $\sigma$-conjugate, we may assume that $b=b_M\in M(\breve F)$. Let $M'\subseteq M$ be the centralizer 
of the $M$-dominant Newton point of $b'_M$. Replacing $b'_M$ by an $M$-$\sigma$-conjugate, we may assume that $b'_M\in M'(\breve F)$. It also induces a reduction of $\E_{b'}$ to $P$.

Let $z\in \pi_{{\rm HT}}^{-1}(x)$ for some $x\in Z(C)$. Then $z$ corresponds to a modification between $\E_b$  and $\E_{b'}$ such that there is a reduction $\E'_{P,z}$ of $\E_{b'}$ to $P$ which by Proposition \ref{lemlamlem} is isomorphic to $\E_{b'}^P$ and such that $\E'_{P,z}$ and the parabolic reduction $\E_b^{P}$ correspond to each other via the modification.  Still by the same proposition, we can extend the isomorphism $\E'_{P,z}\cong\E_{b'}^P$ to an automorphism of $\E_{b'}$ (which we may still compose with automorphisms of $\E_{b'}^P$). Composing the modification corresponding to $z$ with this automorphism of $\E_{b'}$, we obtain a modification $z'$ between $\E_b^{P}$ and $\E_{b'}^P$. We denote the induced element of $\mathcal{M} ( G, \mu^{-1}, b, b')_{\infty}(C)$ again by $z'$.

Let $\lambda'$ be such that $\pi_{\dR}(z')$ is contained in the semi-infinite cell for $\lambda'$, and let 
$$S=\pi_{\dR}^{-1}(\coprod_{\lambda'}S_{\lambda'})$$
where the union is taken over all $\lambda'$ arising for the various $x'$ and choices of automorphisms of $\E_{b'}$.  By Lemma \ref{lemslopekappa}, we have $-\kappa_M(b'_M)-(\lambda')^{\sharp_M}=-\kappa_M(b)$ for all such $\lambda'$. Since $\pi_{\dR}$ is pro-\'etale, we have $\dim \pi_{\dR}(\pi_{\HT}^{-1}(Z))=\dim \pi_{{\rm HT}}^{-1}(Z)$ and $\dim S=\max_{\lambda'} \dim (S_{\lambda'}\cap \Gr_{G,\mu^{-1},b'}^{[b]})$. 

Changing the automorphism of $\E_{b'}$ by an automorphism of $\E_{b'}^P$ replaces $z'$ by another modification between $\E_b^{P}$ and $\E_{b'}^P$ that is also contained in $S$. Thus by the above consideration, $\pi_{\HT}^{-1}(Z)$ is contained in the image of $S\times^{\Aut(\E_{b'}^P)}\Aut(\E_{b'})\rightarrow \mathcal M(G,\mu^{-1},b,b')_{\infty}$ mapping a pair $(z',j)$ to the composition of the modification $z'$ with the automorphism $j$ of $\E_{b'}$.  

Altogether, we obtain
\begin{equation}\label{eqdimest}
\dim \pi_{\dR}(\pi_{\HT}^{-1}(Z))\leq \max_{\lambda'}\dim (S_{\lambda'}\cap \Gr_{G,\mu^{-1}})+\dim \Aut(\E_{b'})-\dim\Aut(\E_{b'}^P).
\end{equation}
We compute the second and third summand on the right hand side using the explicit description of $\Aut(\E_{b'})$ in \cite[III.5.1]{FS}. From loc.~cit.~we obtain $$\dim \Aut(\E_{b'})=\langle 2\rho, \nu_{b'}\rangle=\sum_{\alpha>0}\langle \alpha,\nu_{b'}\rangle$$ and $$\dim \Aut(\E_{b'}^P)=\sum_{\alpha>0,\langle \alpha,\nu_{b'_M}\rangle >0}\langle \alpha,\nu_{b'_M}\rangle.$$ Since $\nu_{b'}$ is the dominant representative of the $W$-orbit of $\nu_{b'_M}$, we have $$\sum_{\alpha>0}\langle \alpha,\nu_{b'}\rangle=\sum_{\alpha>0}|\langle \alpha,\nu_{b'_M}\rangle|.$$ Hence 
\begin{align*}
\dim \Aut(\E_{b'})/\Aut(\E_{b'}^P)&=\sum_{\alpha>0,\langle \alpha,\nu_{b'_M}\rangle<0}|\langle \alpha,\nu_{b'_M}\rangle|\\
&=\frac{1}{2}\sum_{\alpha>0}\left(|\langle \alpha,\nu_{b'_M}\rangle|-\langle \alpha,\nu_{b'_M}\rangle\right)\\
&=\frac{1}{2}\sum_{\alpha>0}\langle \alpha,\nu_b-\nu_{b'_M}\rangle\\
&=\langle\rho,\nu_{b'}-\nu_{b'_M}\rangle.
\end{align*}
We bound the first summand on the right hand side of \eqref{eqdimest} for each $\lambda'$ separately. For this, we use that we are in the minuscule case. Then $\lambda'=\mu^{ww_0}$ for some $w\in W$. The Bialinicky-Birula isomorphism identifies $S_{\lambda'}\cap \Gr_{G,\mu^{-1}}$ with $\F(G,\mu^{-1})^{w,\diamond}$, the Schubert cell for $w$ in the flag variety for $(G,\mu^{-1}_{\dom})$. The dimension of this Schubert cell is equal to $\ell(w)=\langle \rho,\mu^{ww_0}+\mu\rangle$, hence $\dim S_{\lambda'}\cap \Gr_{G,\mu^{-1}}=\langle \rho,\lambda'+\mu\rangle$.

Altogether we have 
$$\dim \pi_{\dR}(\pi_{\HT}^{-1}(Z))= \dim \pi_{{\rm HT}}^{-1}(Z)\leq \langle \rho,\lambda'+\mu\rangle+\langle \rho,\nu_{b'}-\nu_{b'_M}\rangle,$$
and it remains to show  that
\begin{equation*}
 \langle \rho,\lambda'+\mu\rangle+\langle \rho,\nu_{b'}-\nu_{b'_M}\rangle<\langle 2\rho,\mu\rangle.
\end{equation*}

Let $P'$ be the standard parabolic subgroup with Levi factor $M'$ and recall that we assumed $b'_M\in M'(\breve F)$. For some $x\in S_{\lambda'}\cap \Gr_{G,\mu^{-1}}(C)$ consider the modification $\E_{b',x}\cong\E_b$ and the reduction $(\E_{b',x})_{P'}$ induced by the reduction $\E_{b'_M}^{P'}$. Since it is a reduction of the semi-stable bundle $\E_b$, we obtain that its slope vector is $\kappa_{M'}(b'_{M'})-(\lambda')^{\sharp_{M'}}\leq_{M'} -\nu_{b}=0$, where we use that $G$ is adjoint. Hence this difference is a non-positive linear combination of the images of positive coroots (of $M$). Let $\tilde \lambda$ be the $M$-dominant representative in the $W$-orbit of $\lambda'$. Then (for example by Lemma \ref{lems455}(1)) $\kappa_{M'}(b'_{M'})\leq_{M'} {\tilde\lambda}^{\sharp_{M'}}$. Since $\nu_{b'_{M}}$ is central in $M'$ and $\tilde \lambda$ is $M'$-dominant, this implies $\nu_{b'_{M}}\leq \tilde\lambda$. More precisely (since they have the same image in $\pi_1(M)_{\Gamma,\Q}$), their difference is a non-negative rational linear combination of positive coroots of $M$. Let $w\in W$ with $\nu_{b'_M}^w=\nu_{b'}$ dominant. Since $\nu_{b'_M}$ is $M$-dominant, we may choose $w\in W^{M}$. In particular, conjugation by $w$ maps positive coroots in $M$ to positive coroots. Hence a coroot that is a sum of $l$ simple coroots for $M$ (such that the pairing with $\rho$ is $l$) is mapped to a sum of $l$ positive coroots (whose pairing with $\rho$ is then greater than or equal to $l$). Altogether, we obtain
\begin{eqnarray*}
\langle \rho,\lambda'-\nu_{b'_M}\rangle&\leq&\langle \rho,\tilde\lambda-\nu_{b'_M}\rangle\\
&\leq&\langle \rho,w(\tilde\lambda-\nu_{b'_M})\rangle\\
&\leq&\langle \rho,\tilde\lambda_{\dom}-\nu_{b'}\rangle\\
&\leq&\langle \rho,\mu-\nu_{b'}\rangle.
\end{eqnarray*}
If equality holds, the first equality implies that $\lambda'=\tilde \lambda$. Let $\tilde M\supseteq M'$ be the smallest standard Levi subgroup such that $\tilde\lambda-\nu_{b_{M'}}$ is in the $\Q$-vector space generated by the coroots of $\tilde M$. Then equality in the second inequality above means that $\tilde M^{w}$ is again a standard Levi subgroup. Furthermore, $w(\tilde\lambda-\nu_{b'_M})={\tilde\lambda}^w-\nu_{b'}$ is in the $\Q$-vector space generated by the coroots of $\tilde M^{w}$. The other equalities prove that ${\tilde\lambda}^w=\tilde\lambda_{\dom}=\mu$. This is in contradiction to Hodge-Newton indecomposability (for $\tilde M^w$).
\end{proof}
\begin{remark}
Most of the above argument also works for non-minuscule $\mu$. The only place where we need the assumption to be minuscule is the bound on the dimension of $S_{\lambda}\cap \Gr_{G,\mu}$ by $\langle\rho, \lambda+\mu\rangle$. The proof would generalize verbatim to non-minuscule $\mu$ once one can show such a bound for all $\lambda$ and $\mu$.
\end{remark}

\begin{lemma}\label{lems455}
Let $G$ be quasi-split and fix a maximal torus and a Borel subgroup containing it. Let $P_1,P_2$ be standard parabolic subgroups of $G$ with standard Levi factors $M_1,M_2$ and unipotent radicals $N_1,N_2$.
\begin{enumerate}
\item  Let $\nu\in X_*(T)_{\mathbb Q}$ be dominant. For every $w\in W,$ $$\nu\geq w^{-1}(\nu)$$ in $X_*(T)_{\mathbb Q}$.
\item Let $\nu\in X_*(T)_{\mathbb Q}$ be dominant $P_1$-regular. If $w\in {}^{M_1}W^{M_2}$ is such that the images of $\nu$ and $w^{-1}(\nu)$ in $\pi_1(M_2)_{\Q}$ agree, then $w=1$.
\end{enumerate}
\end{lemma}
Here, ${}^{M_1}W^{M_2}$ denotes the subset of elements of $W$ that are shortest representatives of their $W_{M_1}\times W_{M_2}$-double cosets.
\begin{proof}
(1) is shown in \cite[Lemma 4.8]{S} and follows immediately from the assumption that $\nu$ is dominant. For (2) we replace $\nu$ by a suitable multiple and may thus assume that $\nu\in X_*(T)$, and that the images of $\nu$ and $\nu':=w^{-1}(\nu)$ in $\pi_1(M_2)$ agree. Since $\nu$ and $\nu'$ are in the same $W$-orbit and $w^{-1}\in {}^{M_2}W$, there is no root hyperplane for $M_2$ separating the two elements. Therefore, $\nu-\nu'$ is a non-negative linear combination of coroots $\alpha^{\vee}$ for roots $\alpha$ of $T$ in $N_2$. Since $\nu=\nu'$ in $\pi_1(M_2)$, this implies that $\nu=\nu'$, hence $w^{-1}(\nu)\in W_{M_2}(\nu)$. Since $w\in{}^{M_1}W^{M_2},$ this implies $w=1$.
\end{proof}

In particular, we obtain a new proof of the classification of data $(G,\mu,b)$ for which the admissible locus coincides with the weakly admissible locus. This has previously been shown by Chen, Fargues and Shen \cite[Thm.~6.1]{CFS}.
\begin{cor}
Let $b\in G(\breve{\Qp})$ be basic, and let $\{\mu\}$ be a minuscule conjugacy class of cocharacters of $G$. Then $\Gr_{G,\mu}^{\a}=\Gr_{G,\mu}^{\wa}$ if and only if $(G,\mu,b)$ is fully Hodge-Newton decomposable.
\end{cor}
Here, as in \cite{GHN}, a triple $(G,\mu,b)$ is fully Hodge-Newton decomposable if every non-basic element of $\BGmub$ is Hodge-Newton decomposable.
\begin{proof}
We have $\Gr_{G,\mu}^{\a}=\Gr_{G,\mu}^{\wa}$ if and only if the weakly admissible locus does not intersect any non-basic Newton stratum. By the theorem this is the case if and only if every other non-empty Newton stratum is Hodge-Newton decomposable.
\end{proof}

\bibliographystyle{alpha}
\bibliography{references}

\end{document}